\documentclass[draft]{amsart}
\usepackage{amsmath,amssymb,amsthm,bbm}

\newcommand{\hcap}{\operatorname{hcap}} 
\newcommand{\dist}{\operatorname{dist}}
\newcommand{\diam}{\operatorname{diam}}
\newcommand{\ecap}{\operatorname{cap}}
\newcommand{\real}{\operatorname{Re}}
\newcommand{\imag}{\operatorname{Im}}
\newcommand{\inrad}{\operatorname{inrad}}
\newcommand{\Disk}{\mathbb{D}}
\newcommand{\D}{\Disk}
\newcommand{\Half}{\mathbb{H}}
\newcommand{\ball}{\mathcal{B}}
\newcommand{\ee}{\epsilon}
\newcommand{\eps}{\epsilon}
\newcommand{\Prob}{\mathbb{P}}
\newcommand{\PP}{\Prob}
\newcommand{\C}{\mathbb{C}}
\newcommand{\E}{\mathbb{E}}
\newcommand{\R}{\mathbb{R}}
\newcommand{\Z}{\mathbb{Z}}
\newcommand{\one}{\mathbf{1}}
\newcommand{\grid}{\mathcal{D}}
\newcommand{\F}{\mathcal{F}}
\newcommand{\ubs}{\mathcal U}
\newcommand{\edgeset}{\mathcal{E}}
\newcommand{\bd}{\partial}  
\newcommand{\PK}[3]{\lambda({#1},{#2};{#3})}
\newcommand{\PKa}{\lambda}
\newcommand{\LE}[1]{\Lambda\{#1\}}
\newcommand{\RLE}[1]{\overline\Lambda\{#1\} }
\newcommand{\rev}[1]{\overline{#1}}
\newcommand{\gam}{\gamma}
\newcommand{\bs}[1]{\mathcal{S}(#1)}
\newcommand{\vt}{\vartheta}
\newcommand{\fr}{\frac}
\newcommand{\vp}{\varphi}
\newcommand{\iy}{\infty}

\numberwithin{equation}{section}
\theoremstyle{plain}
\newtheorem{theorem}{Theorem}[section]
\newtheorem{proposition}[theorem]{Proposition}
\newtheorem*{proposition*}{Proposition}
\newtheorem{lemma}[theorem]{Lemma}
\newtheorem{corollary}[theorem]{Corollary}

\theoremstyle{remark}
\newtheorem*{remark*}{Remark}
\newtheorem*{definition*}{Definition}

\begin{document}

%
%
%
%

\title[On the rate of convergence of LERW to SLE$_2$]{On the rate of convergence of loop-erased random walk to  SLE$_2$}
\date{\today}
\author[C.~Bene\v{s}]{Christian Bene\v{s}}
\address{C.~Bene\v{s}: Department of Mathematics, Brooklyn College of the City University of New York, Brooklyn, NY 11210-2889  USA}
\email{christian.benes@gmail.com}
\author[F.~Johansson Viklund]{Fredrik Johansson Viklund}
\address{F.~Johansson Viklund: Department of Mathematics, Columbia University, New York, NY 10027-6940 USA}
\email{fjv@math.columbia.edu}
\author[M.~J.~Kozdron]{Michael J.~Kozdron}
\address{M.~J.~Kozdron: Department of Mathematics \& Statistics, University of Regina, Regina, SK S4S 0A2 Canada}
\email{kozdron@stat.math.uregina.ca}
\keywords{Schramm-Loewner evolution, loop-erased random walk, Brownian motion, strong approximation, conformal invariance}
\subjclass[2000]{Primary 60J67; Secondary 82B41, 60J65}

\begin{abstract}
We derive a rate of convergence of the Loewner driving
function for planar loop-erased random walk to Brownian motion with speed
2 on the unit circle, the Loewner driving function for radial SLE$_2$. 
The proof uses a new estimate of the difference between the discrete
and continuous Green's functions that is an improvement
over existing results for the class of domains we consider.
Using the rate for the driving process convergence
along with additional information about SLE$_2$, we also obtain a rate of
convergence for the paths with respect to the Hausdorff distance.
\end{abstract}

\maketitle

\section{Introduction}\label{Intro}

The Schramm-Loewner evolution (SLE) is a one-parameter family of random 
planar growth processes constructed by solving the Loewner equation when the driving function is a one-dimensional Brownian motion. SLE was introduced by Schramm in~\cite{schramm} and has been shown to describe the scaling limits of a number of  two-dimensional discrete models from statistical mechanics including the Ising model~\cite{chelkak_smirnov,HDC-smirnov}, percolation~\cite{smirnov_perc,smirnov_CR}, loop-erased random walk~\cite{LSW-aop}, uniform spanning trees~\cite{LSW-aop}, the harmonic explorer~\cite{SchShef}, the discrete Gaussian free field~\cite{SchShef2}, and the $q=2$ random cluster model~\cite{smirnov_AnnMath}.  This has provided a means for developing a rigorous mathematical understanding of these models. SLE has also allowed a number of
long-standing open problems about planar Brownian motion to be solved,
notably, Mandelbrot's conjecture about the Hausdorff dimension of the
Brownian frontier~\cite{LSW-43} and the determination of the values of the 
Brownian intersection exponents~\cite{LawSW1,LawSW2,LawSW3}. Despite a rapid progress in the understanding of
questions involving SLE, there are still several fundamental open
problems. Some of these were communicated by Schramm
in~\cite{schrammICM}, in particular that of ``obtain[ing] reasonable
estimates for the speed of convergence of the discrete processes which
are known to converge to SLE.'' One of the motivations for this
question, besides its being of independent interest, is that results of
this type could lead to improved estimates for certain critical exponents; see the discussion in~\cite{schrammICM}.

The loop-erased random walk is a self-avoiding random walk obtained by chronologically erasing the loops of a simple random walk.
It was proved by Lawler, Schramm, and Werner
in~\cite{LSW-aop} that the scaling limit of loop-erased random walk in
a simply connected domain is SLE$_2$. Arguably the most important step in this proof is to show that the Loewner driving
function for the loop-erased random walk path converges to Brownian motion with speed 2, the Loewner driving function for SLE$_2$.

The primary purpose of this paper is to establish a rate for this convergence. To the best of our knowledge, this is the first instance of a formal derivation of a rate of convergence for any of the discrete processes
known to converge to SLE, although Smirnov has given without proof a rate of convergence for critical percolation crossing probabilities; see Remark~3 of~\cite{smirnov_perc}.

We believe that the results obtained in this paper could be very helpful in proving convergence rates for other models whose scaling limit is SLE and that much of the work is more or less directly applicable. Essentially only Section~\ref{MGsect}, where we obtain a rate of convergence for the particular martingale observable we used, is specific to loop-erased random walk. At the end of Section~\ref{HM_ROC_sect}, we briefly discuss an application of our methods to the FK-Ising model.

\subsection{Statement of the main result and outline of the paper}
Let $D \subsetneq \C$ be a simply connected domain with $0 \in D$, and
let $\psi_D: D \to \Disk$, where $\Disk$ is the open unit disk, be the unique conformal map with $\psi_D(0)=0$, 
$\psi_D'(0)>0$. Let $D^n$ be the $n^{-1}\Z^2$ grid domain approximation of $D$; that
is, the connected component containing $0$ in the complement of the closed faces of
$n^{-1} \Z^2$ intersecting $\partial D$. Let $\gamma^{n}$ denote the time-reversal of loop-erased random walk on $n^{-1} \Z^2$ started at $0$
and stopped when hitting $\bd D^n$. Note that $D^n$ is simply connected and let $\psi_{D^n}: D^n \to \D$ be the conformal map normalized as above. Let 
\[
W_{n}(t)=W_n(0)e^{i  \vartheta_n(t)}, \quad t \ge 0,
\]
denote the Loewner driving function for the curve $\tilde \gamma^n=\psi_{D^n}(\gamma^{n})$ parameterized by capacity and let $\inrad(D) = \sup\{r:\ball(0,r) \subset D\}$. The following
is our main result.

\begin{theorem}\label{the.theorem2}
Let $0<\epsilon<1/24$ be fixed, and let $D$ be a simply connected domain with $\inrad(D) = 1$. For every $T> 0$ there exists an
$n_0< \infty$ depending only on $T$ such that whenever $n > n_0$ there is a 
 coupling of $\gamma^{n}$ with Brownian motion $B(t)$,  $t\ge 0$, where $e^{iB(0)}$ is uniformly distributed on the unit circle, 
 with the property that
\begin{equation}\label{the.equation2}
\Prob \left( \sup_{0 \le t \le T} |W_n(t)-e^{iB(2t)}| > n^{-(1/24-\ee)} \right)  < n^{-(1/24 -\ee)}.
\end{equation}
\end{theorem}

The proof of Theorem~\ref{the.theorem2}, which follows the general
strategy that was outlined in~\cite{schramm} and implemented in detail for the proof of convergence in~\cite{LSW-aop}, has four main components. Each of
these is covered in a separate section following
Section~\ref{Sect-notation}, in which we introduce some notation and
preliminary results. 

In Section~\ref{HM_ROC_sect} we derive a rate of convergence for the
distribution of the conformal image of the starting point of the time-reversed loop-erased random walk. This result, which is given in
Proposition~\ref{HM_ROC}, is an application of the strong approximation of
simple random walk due to Koml\'os, Major, and Tusn\'ady; see~\cite{kmt2}. 

Section~\ref{MGsect} covers the second main step of our proof which is to derive a
rate of convergence of the martingale observable for the
loop-erased random walk path. We use the same observable as in~\cite{LSW-aop}, but our method for proving convergence is
somewhat different. More precisely, we improve estimates of the
discrete Green's function from~\cite{KozL1} which, together with geometric
arguments, yield a rate of convergence for the observable; see
Theorem~\ref{MGthm}. Certain technical details in this section are deferred to the appendices. 

Next, in Section~\ref{keysect}, the Loewner equation is used to
transfer information about the observable to information about the
Loewner driving function for a piece of the loop-erased random walk path. In
particular, it is shown that the driving function is, up to explicit
error terms, a martingale on a certain mesoscopic scale that depends
on the rate of convergence of the observable; see Proposition~\ref{keyestimate}. 

The estimates from the previous sections, a sharp martingale maximal inequality, and the Skorokhod embedding theorem are then used in  Section~\ref{Sect-theproof} 
to find a coupling of the driving function with Brownian motion such that~\eqref{the.equation2} holds. This step concludes the
proof of Theorem~\ref{the.theorem2}. 

An informal summary of the main steps in our derivation of a rate of convergence is given in Section~\ref{discuss} where we also  discuss the optimality of the rate we obtained.

Finally, in Section~\ref{Sect-hausdorff} we use Theorem~\ref{the.theorem2} and a derivative estimate for radial SLE$_2$ (which we derive from an estimate for chordal SLE$_2$) to obtain a rate of convergence for the paths with respect to Hausdorff distance; see Theorem~\ref{hd.thm} for a precise statement.

In Appendix~\ref{greenappendix} we prove Theorem~\ref{main_green_thm} which is a particular Green's function estimate. This uniform estimate for the expected number of visits to $x$ by two-dimensional simple random walk starting at 0 before exiting a simply connected grid domain is an improvement of Theorem~1.2 from~\cite{KozL1} and is a result of independent interest.  It is then used in Appendix~\ref{prop42appendix} to help prove Proposition~\ref{KLtheorem}.

\section{Notation and preliminaries}\label{Sect-notation}

We now introduce the notation that will be used throughout this paper. General information about the basics of SLE and much of the necessary background material can be found in~\cite{SLEbook}. To facilitate the reading of this paper we have tried to be consistent with the notation used in~\cite{LSW-aop}. 

\subsection{Conformal maps and grid domains}\label{conformal}

Suppose that $\C$ denotes the complex plane, and write $\Disk =\{z: |z|<1\}$ for the unit disk in $\C$.  We write $\ball(z,r)=\{w : |w-z| < r\}$ and use the notation $\mathcal{A}(r,R)$ to denote the annulus $\{z: r < |z| < R\}$. For a set $D \subset \C$, we define the \emph{inner radius} of $D$ with respect to $z\in\C$ to be
\(
\inrad_z(D) = \sup\{r:\ball(z,r) \subset D\}
\)
and we write $\inrad(D)$ for $\inrad_0(D)$. 
We say that a domain $D \subset \C$ is a \emph{grid domain (with respect to $\Z^2$)}
if the boundary of $D$ consists of edges of the lattice $\Z^2$, and we
write $\grid$ for the set of all simply connected grid domains $D$
such that $0<\inrad(D)<\infty$; that is, those simply connected grid
domains $D \neq \C$ such that $0 \in D$. 

If $D$ is a simply connected domain containing the origin, we denote
by $\psi_D$ the unique conformal map of $D$ onto $\Disk$ with
$\psi_D(0)=0$ and $\psi_D'(0)>0$. 

For simplicity, we will call the value $\log \psi'_D(0)$ the \emph{capacity} (from 0) of $D^c$ and we denote this by $\ecap(D^c)$. In the particular case when $D=\D \setminus K$ for some compact set $K$,
we write $\ecap(K)$. 

We will also have occasion to consider a similar quantity for the upper half-plane $\Half$. Suppose that $K \subset \Half$ is bounded. If $K=\Half \cap \overline{K}$ and $\Half\setminus K$ is simply connected, then there exists a unique conformal transformation $\psi_K:\Half\setminus K \to \Half$ such that
$$\lim_{z \to \infty} [ \psi_K(z)-z] =0.$$
The \emph{half-plane capacity} (from infinity) of $K$ is defined by
$$\hcap(K) = \lim_{z \to \infty} z[ \psi_K(z)-z].$$
See Chapter~3 of~\cite{SLEbook} for further details.

We say that a proper subset  $A \subset \Z^2$  is \emph{connected} if every 
two points in $A$ can be connected by a nearest neighbor path
 staying in $A$, and is called \emph{simply connected} 
if both $A$ and $\Z^2 \setminus A$ are connected. The \emph{boundary} of $A$ 
is given by $\bd A = \{y \in \Z^2 \setminus A: |y-x| = 1 \text{ for
  some } x \in A \}$.

When $D$ is a grid domain, we
write $V(D)=D \cap \Z^2$ for the lattice points contained in
$D$ and note that $\bd V(D)$ is contained in $\bd D \cap \Z^2$.

If $D$ is a simply connected domain with a Jordan boundary, it is well-known that $\psi_D$ can be extended continuously to the boundary so that if $u \in \bd D$, then
 $\psi_D(u)=e^{i\theta_D(u)} \in \bd \Disk$. For our purposes, we will
 be concerned with grid domains which may not have a Jordan
 boundary. This means that if $D \in \grid$, then a boundary point
 may correspond under conformal mapping to several points on the boundary of the unit
 disk. To avoid using prime ends (see~\cite{pommerenke} for a full discussion), we adopt the
 convention from~\cite{LSW-aop} of viewing the boundary of $\Z^2
 \cap D$ as
 pairs $(u,e)$ of a point $u \in \bd D \cap \Z^2$ and an incident edge $e$ connecting $u$ to a neighbor in $D$. We write
 $V_{\partial}(D)$ for the set of such pairs, and if $v \in
 V_{\partial}(D)$, then the notation $\psi_D(v)$ means
 $\lim_{z \to u}\psi_D(z)$ along $e$, and this limit always exists. If $v=(u,e) \in V_{\partial}(D)$, then
 we write $A_v$ for the neighbors $w$ of $u$ such that the edge
 $(w,u)$ corresponds to the same limit on $\partial \D$ as $e$. 

Furthermore, if $B$ is planar
 Brownian motion and $T_D=\inf\{t \ge 0 : B_t \not \in D\}$ is the exiting time of $D$ by $B$, where $D$
 is a grid domain, it is
 known that the limit $\psi_D(B_{T_D})=\lim_{t \to T_D} \psi_{D}(B_t)$ exists
 almost surely.    
 
Finally, recall that if $D$ is a simply connected domain, then a \emph{crosscut} of $D$ is an open Jordan curve $C$ in $D$ such that $\overline{C}=C\cup \{w,z\}$ with $w$, $z \in \bd D$; note that $w$ and $z$ need not be distinct. For further details,  see~Section~2.4 of~\cite{pommerenke}.
 
 We state here Koebe's well-known distortion, growth, and one-quarter theorems. We will use these results extensively. See~\cite{pommerenke} for further discussion and proofs. We remark that we usually refer to the first two sets of inequalities as the Koebe distortion theorem.  

 \begin{lemma}\label{Koebedistthm}
Let $D$ be a simply connected domain and suppose $f: D \to \C$ is a conformal map. Set
$d=\dist(z, \partial D)$ for $z \in D$. If $|z-w| \le rd$ for some $0<r<1$, then
\begin{equation*}
\fr{1-r}{(1+r)^3}|f'(z)| \le |f'(w)| \le \fr{1+r}{(1-r)^3}|f'(z)|,
\end{equation*}
\begin{equation*}
\fr{|f'(z)|}{(1+r)^2}|z-w| \le |f(z)-f(w)| \le \fr{|f'(z)|}{(1-r)^2}|z-w|,
\end{equation*}
and 
\begin{equation*}
\label{quarter}
\mathcal{B}(f(z), d|f'(z)|/4) \subset f(D),
\end{equation*}
where $\mathcal{B}(w,\rho)$ denotes the open disk of radius $\rho$ around $w$.
 \end{lemma} 

Let $d_f(z)=\dist(f(z), \partial D')$, where $f$ is conformal and $D'=f(D)$.
The following result, which we will refer to as Koebe's estimate, is a consequence of the last lemma, see~\cite{garnett_marshall}:
 \begin{equation*}
 \frac{1}{4}d |f'(z)| \le d_f(z) \le 4 d |f'(z)|.
 \end{equation*}
 We will also make use of various versions and consequences of
the Beurling projection theorem, in both the continuous and discrete
setting. We state three versions here; see~\cite{SLEbook},~\cite{kesten}, and~\cite{benesnotes}, respectively. 

\begin{lemma}\label{DRbeurling}
Let $D$ be a simply connected domain, and let $\vp : D \to \D$ be a conformal map with $\vp(0)=0$. If $\beta$ is a
simple curve in $D$ with one endpoint on $\partial D$, then there
exists a constant $c < \infty$ such that
\begin{equation*}
\diam  \vp(\beta) \le c \left[\frac{\diam  \beta}{ \inrad(D)} \right]^{1/2}.
\end{equation*} 
\end{lemma}

Throughout this paper, for any set $A\subset\C, T_A=\inf\{t \ge 0 : B_t \not \in A\}$ will denote the exiting time of $A$ by Brownian motion $B$ and $\tau_A=\inf\{k \ge 0 : S_k \not \in A\}$ will denote the exiting time of $A$ by random walk $S$.

\begin{lemma}\label{beurling2}
There exists a constant $c >0$ such that for any $R \geq 1$, any
  $x\in\C$ with $|x|\leq R$, any $A\subset \C$ with $[0,R] \subset \{|z|:z\in A\}$,
\begin{equation*}
\Prob^x(\xi_R \leq T_{A^c})\leq c\left(|x|/R\right)^{1/2},
\end{equation*}
where $\xi_R = \inf\{t\geq 0:|B_t|\geq R\}$ and $T_{A^c} = \inf\{t\geq
0:B_t\in A\}$ is the first hitting time of $A$ by $B$, where $B$ is planar Brownian motion.
\end{lemma}

\begin{lemma}\label{beurling3}
There exists a constant $c >0$ such that for any $n \geq 1$, any
  $x\in\Z^2$ with $|x|\leq n$, any connected set $A\subset \Z^2$
  containing the origin and such that $\sup\{|z|:z\in A\}\geq n$,
$$\Prob^x(\Xi_n \leq \tau_{A^c})\leq c\left(|x|/n\right)^{1/2},$$
where $\Xi_n = \inf\{k\geq 0:|S_k|\geq n\}$ and $\tau_{A^c} = \inf\{k\geq
0:S_k\in A\}$ is the first hitting time of $A$ by $S$, where $S$ is simple random walk on $\Z^2$.
\end{lemma}

Note that in the last two lemmas it would be more natural to define $T_A$ and $\tau_A$ as hitting times rather than exiting times of $A$, as the statements could be slightly simpler than as we wrote them, but we are using this notation to be consistent with the rest of the paper, where thinking in terms of exiting times will be more natural.

\subsection{Green's functions}\label{greenZ}

If $D$ is a domain whose boundary includes a curve, 
let $g_D(z,w)$ denote the Green's function for $D$. 
If $z \in D$, we can define $g_D(z, \cdot)$ 
as the unique harmonic function on $D\setminus \{z\}$, 
vanishing on $\bd D$ (in the sense that $g_D(z,w) \to 0$ as $w \to w_0$
for every regular $w_0 \in \bd D$), with 
\begin{equation*}
g_D(z,w) = -\log|z-w| + O(1) \;\text{ as }\; |z-w| \to 0. 
\end{equation*} 
 In the case $D = \Disk$, we have
\begin{equation}\label{defngreen}
g_{\D}(z,w)=\log \left|\overline{w}z-1\right| - \log \left|w-z\right|.
\end{equation}
Note that $g_{\D}(0,z) = -\log |z|$ and $g_{\D}(z,w)=g_{\D}(w,z)$. 
An equivalent formulation of the Green's function 
can be given in terms of Brownian motion, namely
$g_D(z,w) = \E^z[\log|B_{T_D}-w|] - \log |z-w|$
for distinct points $z$, $w \in D$ where $T_D = \inf \{t : B_t \not\in D \}$. 
The Green's function is a well-known example of a conformal invariant; see
Chapter~2 of~\cite{SLEbook} for further details.
Note that the conformal map $\psi_D: D \to \D$ can be written as
\begin{equation*}
\psi_D(z) = \exp\{-g_D(z) + i\theta_D(z)\}, \;\;\; z \in D, 
\end{equation*}
where $g_D(z) =g_D(0,z)$ and $\theta_D(z) = \arg(\psi_D(z))$.  In particular, we can write $g_D(z)=-\log|\psi_D(z)|$.

Thus, suppose $D \in \grid$ is a grid domain with $\inrad(D)=R$. 
If $z \in D$ with $\dist(z,\bd D)=1$, then by a Beurling estimate $g_D(z)=O(R^{-1/2})$, and if $u \in V_{\partial}(D)$ and $z \in  A_u$, then $|\psi_D(u) - \psi_D(z)| = O(R^{-1/2})$ so that
\begin{equation}\label{thetaestimate}
\theta_D(u) = \theta_D(z) + O(R^{-1/2})
\end{equation}
in the sense that for each $u$ as above, we can choose a branch such that~\eqref{thetaestimate} holds.

Suppose that $S$ is a simple random walk on $\Z^2$ and
 $A$ is a proper subset
of  $\Z^2$. If $\tau_A = \min \{j \ge 0 : S_j \not\in A\}$,
 then  we let
\begin{equation*}\label{june7.2}
G_A(x,y) 
= \sum_{j=0}^{\infty} \Prob^x (S_j=y,\, \tau_A > j )
\end{equation*}
denote the Green's function for random walk on $A$. Note that $G_A(x,y)=G_A(y,x)$, and set $G_A(x) = G_A(x,0) = G_A(0,x)$.  In analogy with the Brownian motion case,
 we have
\begin{equation}\label{green1.1}
G_A(x,y)=\E^x[a(S_{\tau_A}-y)]-a(y-x) \, \text{ for }\, x, y\in A
\end{equation}
where $a$ is the potential kernel for simple random walk defined by
\[a(x)  =  \sum_{j=0}^\infty \left[\Prob^0(S_j=0) - \Prob^x( S_j=0)\right].\]
For details, see Proposition~1.6.3 of~\cite{LawlerGreen}.
It is known (see~\cite{FukU1} for details) that 
\begin{equation}\label{1.2}
a(x) = \frac{2}{\pi}\log|x| + k_0 +O(|x|^{-2})
\end{equation}
as $|x| \to \infty$ where $k_0 = (2\varsigma + 3\ln 2)/\pi$ and $\varsigma$ is Euler's constant.
 
Let $x \in A$, $w=(u,e) \in V_{\bd}(A)$, and recall that $A_w$ is defined to be the set of neighbors $y \in A$ of $u$ such that the edge $(y,u)$ and $e$ correspond to the same limit on $\partial \D$. We say that these edges correspond to $w$. We define $H_A(x,w)$ to be the probability that a simple random
walk starting at $x$ exits $A$ at $u$ using one of the edges determined by the vertices in $A_w$; that is, through one of the edges corresponding to $w$. If a boundary vertex $u$ corresponds to only one limit on $\partial \D$ we simply write $H_A(x,u)$. We note that a so-called
last-exit decomposition implies the identity
\begin{equation}\label{GFdecomp}
H_A(x,w) = \frac{1}{4} \sum_{A_w} G_A(x,y).
\end{equation}

\subsection{Loop-erased random walk}

We now briefly review the definition of loop-erased random walk. Further details may be found in Chapter~7 of~\cite{LawlerGreen}. The following loop-erasing 
procedure, which works for any finite simple random walk path in $\Z^2$, 
 assigns a self-avoiding path to each such random walk  path.
 
 Suppose that  $S = S[0,m] = [S_0, S_1, \ldots, S_m]$ is a simple
 random walk path of length $m$.   The loop-erased part of $S$, denoted $\LE{S}$, is constructed recursively as follows. If $S$ is already self-avoiding, 
set $\LE{S}=S$.  Otherwise, let $s_0 = \max\{j : S_j=S_0\}$, and for $i > 0$,
 let $s_i = \max\{j : S_j = S_{s_{i-1}+1} \}$. If we let $n = \min\{i : s_i=m\}$, then 
$\LE{S} = [S_{s_0}, S_{s_1}, \ldots, S_{s_n}]$. Observe that $\LE{S}(0)=S_0$ and $\LE{S}(s_n) = S_m$; that is, the loop-erased random walk has the same starting and ending points as the original simple random walk.

Also notice that the loop-erasing algorithm depends on the order of the points. 
If $a=[a_0,a_1, \ldots,a_k]$ is a lattice path, write $\rev{a} = [a_k,a_{k-1}, \ldots,a_0]$ for its reversal. Thus, if we define reverse loop-erasing by $\RLE{S} = \rev{\LE{\rev{S}}}$, then one can construct a path $S$ such that $\LE{S} \neq \RLE{S}$.  It is, however, a fact that both $\LE{S}$ and $\RLE{S}$ have the same distribution; see Lemma~3.1 of~\cite{LSW-aop}.  As such, since both $\RLE{S}$ (i.e., the time-reversal of loop-erased random walk) and $\LE{\rev{S}}$ (i.e., the loop-erasure of the time-reversal of  random walk) have the same distribution, we will not distinguish between the two and simply say that $\gamma$ is the \emph{time-reversal of loop-erased random walk} if $\gamma = \LE{\rev{S}}$.

In this paper, we will consider the loop-erasure of simple random walk started at $0$ and stopped when hitting the boundary of some fixed grid domain $D$. We call this loop-erased random walk in $D$. 

Loop-erased random walk has the important \emph{domain Markov property}, as is further discussed in Lemma~3.2 of~\cite{LSW-aop}. Suppose $\gamma=(\gamma_0, \gamma_1, \ldots, \gamma_l)$ is the loop-erasure of the time-reversal of a simple random walk that is started from $0$ and stopped when exiting $D$. Then if we condition on the first $j$ steps of $\gamma$, the distribution of the rest of the curve is the same as the time-reversal of loop-erased random walk in $D \setminus \gamma[0,j]$ conditioned to start at $\gamma_j$; that is, it is distributed as the loop-erasure of the time-reversal of a simple random walk started from $0$ conditioned to exit $D \setminus \gamma[0,j]$ through an edge ending in $\gamma_j$. 

\subsection{The Loewner differential equation and Schramm-Loewner evolution}

Suppose the unit disk $\Disk$ is slit by a non self-intersecting curve
$\gamma$ in a way such that $\Disk \setminus \gamma$ is simply
connected and contains $0$. Then we may parameterize the curve by capacity; that is, we choose a parameterization $\gamma(t)$ so that
the normalized conformal map
$g_t: \Disk \setminus \gamma[0,t] \to \Disk$ satisfies
\begin{equation*}
g_t(z)=e^{t}z + O(z^2),
\end{equation*}
around the origin for each $t \ge 0$. It is a theorem by Loewner that
the \emph{Loewner chain} $(g_t)$, $t \ge 0$, satisfies the Loewner differential equation 
\begin{equation}\label{LODE}
\partial_tg_t(z) = g_t(z) \frac{\xi(t)+g_t(z)}{\xi(t)-g_t(z)}, \quad g_0(z)=z,
\end{equation}  
where $\xi(t)=g_t(\gamma(t))$ is a unique continuous unimodular
function. 
The inverse $f_t=g_t^{-1}$ satisfies the partial differential equation
\begin{equation*}
\partial_t f_t(z)=zf'_t(z) \frac{z+\xi(t)}{z-\xi(t)}, \quad f_0(z)=z.
\end{equation*}

Conversely, consider a function continuous on $[0, \infty)$, taking values in
$\partial \Disk$. It follows that~\eqref{LODE} can be solved up to time
$t$ for all $z$ outside $K_t=\{w : \tau(w) \le t\}$, where
$\tau(w)$ is the blow-up time when $g_t(w)$ hits $\xi(t)$; see~\cite{SLEbook} for precise definitions. We note that $g_t$ maps $\Disk \setminus
K_t$ conformally onto $\Disk$ for $t \ge 0$, and that $K_t$ is
called the hull of the Loewner chain. The function $\xi$ is called the driving function for the Loewner chain $(g_t)$ (or $(f_t)$). 

If the limit  
\begin{equation*}
\gamma(t)=\lim_{r \to 1-}f_t(r\xi(t))
\end{equation*} 
exists for $t > 0$ and $t \mapsto \gamma(t)$ is continuous, we say that $(g_t)$
is generated by a curve. In this case, the connected components of $\D \setminus \gamma[0,t]$ and $\D \setminus K_t$ that contain the origin are the same. 

By taking $\xi(t)=\exp\{i B(\kappa t) \}$, where
$B(t)$ is standard Brownian motion and $\kappa >0$, we obtain \emph{radial Schramm-Loewner
evolution with parameter $\kappa$}, or radial SLE$_{\kappa}$ for
short. It is known that SLE$_\kappa$ is generated by a curve; see~\cite{LSW-aop} and~\cite{rohde_schramm}. 

We will also need to work with the chordal version of the Loewner equation which uses the upper half-plane $\mathbb{H}$ as uniformizing domain. In this case the equation reads
\[
\partial_t g_t(z) = \frac{2}{g_t(z)-\xi(t)}, \quad g_0(z)=z, \quad z \in \mathbb{H},
\]
where $\xi(t)$ is the (chordal) real-valued driving function. Similar properties hold in this case, too, but notice that the normalization is at $\infty$, which is a boundary point of $\mathbb{H}$. We refer the reader to 
\cite{SLEbook} for more details. This reference also contains a discussion (see, e.g., p.\ 94) of the reverse-time Loewner equation (or ``inverse flow'' as it is called there), which we shall use in Section~\ref{Sect-hausdorff}. 
\section{A rate of convergence for discrete harmonic measure}\label{HM_ROC_sect}

In this short section we prove a rate of convergence for the boundary hitting
distribution of simple random walk in a grid domain $D$ as the
inner radius increases. Our goal is to give a quantitative statement of the fact that the
image of the starting point of the time-reversed loop-erased random walk
path is close to uniform on $\partial \D$, when the inner radius of $D$ is large. 

As before, for a grid domain $D\in\grid$ we let $\psi=\psi_D$ denote the
conformal map from $D$ onto $\D$ such that $\psi(0)=0$, $\psi'(0)>0$.
We let $\tau_D$ and $T_D$ denote the hitting times of $\partial
D$ for simple random walk $S$ on $\Z^2$ and planar Brownian motion $B$,
respectively. Our goal is now to prove the following result, 
the proof of which is similar to that of Proposition~3.3 of~\cite{KozL1}. 

\begin{proposition}\label{HM_ROC} Let $0<\ee<1/4$ be fixed.
Let $D\in\grid$ be a grid domain. Let $S$ 
denote simple random walk on $\Z^2$ and let $B$ denote planar Brownian motion, both started from $0$. There exists $R_0 <\iy$ such that if $R= \inrad(D)$ and if $R > R_0$,
then there is a coupling of $S$ and $B$ such
that
\[
\PP\left(|\psi(S_{\tau_D})-\psi(B_{T_D})| > R^{-(1/4-\ee)}\right) < R^{-1/4}.
\] 
\end{proposition}

Recall that $\psi(B_{T_D})$ is uniformly distributed on $\partial \D$.
Note that $S_{\tau_D}$ is viewed as an element of $V_{\partial}(D)$;
see Section~\ref{conformal}.
Since the error term in Proposition~\ref{HM_ROC} only depends
on the inner radius, the result when applied to the $n^{-1}\Z^2$ approximation
of a given simply connected domain is independent
of the boundary regularity of the domain that is being approximated.

To prove Proposition~\ref{HM_ROC} we shall use the strong approximation 
of Koml\'os, Major, and Tusn\'ady in a form given in~\cite{KozL1}. The approximation in this form is a non-trivial consequence of the more general Theorem \ref{kmt} below. A technical difficulty is that the joint process $(S,B)$ does not have the Markov property in this coupling, although, of course, each of $S$ and $B$ separately has it. This is the reason for introducing the stopping times $\nu_B$ and $\nu_S$ in the proof below.
In the following, $S$ is defined by linear interpolation for non-integer $t$.  
\begin{lemma}\label{KMTthm1}
There exists $c_0<\iy$ and a coupling of planar Brownian motion $B$ and simple
  random walk $S$ on $\Z^2$, both started from $0$, such that
\begin{equation}\label{KL_KMT}
\PP\left(\sup_{0\le t \le \sigma_R}|S_t-B_t/\sqrt{2}| \ge c_0 \log R \right) = O(R^{-10}),
\end{equation}
where 
\[
\sigma_R=\inf \left\{t: \min\left\{\sup_{0 \le s \le t}|S_s|, \sup_{0 \le s \le t}|B_s|\right\} \ge R^8\right\}.
\]
\end{lemma}
Note that if $R=\inrad(D)$, then in view of Lemma~\ref{beurling2}
\begin{equation}\label{sigma_large}
\PP(\sigma_R < T_D) = O(R^{-7/2})
\end{equation}
and similarly for $\tau_D$.

\begin{proof}[Proof of Proposition~\ref{HM_ROC}] Write $R=\inrad(D)$ and
 let $S$ and $B'$ be simple random walk on $\Z^2$ and planar Brownian motion,
 respectively, both started from $0$. By Lemma~\ref{KMTthm1} we may
 couple $S$ and $B'$ so that~\eqref{KL_KMT} holds. Set $B=B'/\sqrt{2}$ and define
\[
\eta=\inf\{t \ge 0: \min\{\dist(S_t, \partial D), \dist(B_t, \partial D)\} \le 2c_0 \log R\},
\]
where $c_0$ is the constant from Lemma~\ref{KMTthm1}.
Let 
\[
\mathcal{E}_1 = \{|S_{\eta}-B_{\eta}| \le c_0 \log R\} \cap \left\{\sup_{0\le t \le \sigma_R}|S_t-B_t| < c_0 \log R \right\}.
\] 
Then it follows from~\eqref{KL_KMT} and~\eqref{sigma_large} that $\PP(\mathcal{E}_1^{c})=O(R^{-7/2})$. Indeed, since $\eta < T_D$ we have
\[
\PP(\mathcal{E}_1^{c})  \le \PP(\mathcal{E}_1^{c}, \, \sigma_R < T_D) + \PP(\mathcal{E}_1^{c}, \, \sigma_R \ge \eta) 
 \leq  O(R^{-7/2}) + O(R^{-10}).
\]
Define the stopping times
\begin{align*}
\nu_{B} &=\inf\{t \ge 0 : \dist(B_t, \partial D) \le 3c_0\log R \} \;\; \text{and}\\
\nu_{S} &=\min\{j \ge 0 : \dist(S_j, \partial D) \le 3c_0\log R \}.
\end{align*}
On $\mathcal{E}_1$ we clearly have
$\max\{\nu_{B}, \nu_S\} \le \eta < \min\{T_D, \tau_D\}$. 
Let $0 < \alpha < 1$ and let $\mathcal{E}^B_2 \subset \mathcal{E}_1$ be the event that $\mathcal{E}_1$ occurs and that $B$ hits $\partial D$ before exiting the ball $\mathcal{B}(B_{\nu_B}, 3c_0R^{\alpha}\log R)$. By using the strong Markov property of $B$ together with Lemma~\ref{beurling2} we see that $\PP(\mathcal{E}^B_2)\ge 1- O(R^{-\alpha/2})$. Let $Q_B$ be the component of $\mathcal{B}(B_{\nu_B}, 4c_0R^{\alpha}\log R) \cap D$ that contains the point $B_{\nu_B}$. On the event $\mathcal{E}_2^B$ we have that $|B_{\eta}-S_{\eta}| \le c_0\log R$ and that $Q_B$ contains a ball of radius $(3/2)c_0\log R$ around $B_{\eta}$. In particular, $Q_B$ contains $B_{\eta}, S_{\eta}$, and $B_{T_D}$.

We define the event $\mathcal{E}_2^S$ and the set $Q_S$ by replacing $B$ with $S$ in the last paragraph. By the strong Markov property of $S$, using Lemma~\ref{beurling3} we have that $\PP(\mathcal{E}^S_2)\ge 1- O(R^{-\alpha/2})$. On the event $\mathcal{E}_2^B \cap \mathcal{E}_2^S$ the set $Q_B \cap Q_S$ is non-empty and contains the points $B_\eta$ and $S_{\eta}$. 

Consequently, with probability at least $1-O(R^{-\alpha/2})$, the pair of boundary hitting points (in the sense of prime ends) $B_{T_D}$ and $S_{\tau_D}$ can be separated from $0$ in $D$ by a crosscut with length at most $c R^{\alpha}\log R$ for a constant $c <\infty$. Let $\beta$ be such a crosscut 
and let $F \subset \partial D$ be
the part of $\partial D$ that is separated from $0$ by $\beta$. By Lemma~\ref{DRbeurling}, if $R$ is sufficiently large, the harmonic
measure of $F$ from $0$ in $D$ is bounded above by 
$c(R^{\alpha-1}\log R)^{1/2}$ for some constant $c<\iy$. Hence, by
conformal invariance of harmonic measure, the length of the interval
$I=\{z \in \partial \D: \psi^{-1}(z)\in F\}$ satisfies the same
bound (with a different constant). Since $\psi(S_{\tau_D})$ and $\psi(B_{T_D})$ both are contained in $I$, the proof is completed by choosing $\alpha$ such that $(\alpha-1)/2=-\alpha/2$; that is, choose $\alpha=1/2$. 
\end{proof}

\begin{remark*}
We believe that Proposition~\ref{HM_ROC} and part of the work from Section~\ref{keysect} should be the main elements needed to derive a convergence rate for Smirnov's observable for the FK-Ising model; see~\cite{smirnov_ICM}. This, together with the work done in Sections~\ref{keysect},~\ref{Sect-theproof}, and~\ref{Sect-hausdorff} should suffice to give a rate of convergence for the FK-Ising model.
\end{remark*}

\section{A rate of convergence for the martingale observable}\label{MGsect}

The purpose of this section is to provide a rate of convergence for
the martingale observable. This result is given in Theorem~\ref{MGthm}
and will then be used  in Section~\ref{keysect}. Recall that if $D \in
\grid$ is a grid domain, then $\psi_D: D \to \D$ is the conformal map
of $D$ onto $\D$ satisfying  $\psi_D(0)=0$, $\psi_D'(0)>0$. 

\begin{theorem}\label{MGthm}
Let $0<\epsilon <1/4$ and let $0<\rho<1$ be fixed. There exists $R_0 < \infty$ such that the following holds.
Suppose that $D \in \grid$ is a  grid domain with $\inrad(D)=R$, where $R > R_0$.  Furthermore, suppose that $x \in D \cap \Z^2$ with $|\psi_D(x)| \le \rho$ and $u \in V_{\bd}(D)$. If both $x$ and $u$ are accessible by a simple random walk starting from $0$, then
\begin{equation}\label{MGthm.eq}
\frac{H_{D}(x,u)}{H_{D}(0,u)} = \frac{1-|\psi_{D}(x)|^2}
{|\psi_{D}(x)-\psi_{D}(u)|^2} \cdot [\; 1 + O(R^{-(1/4-\epsilon)})  \;].
\end{equation}
\end{theorem}

The proof is  given in Section~\ref{MGthmsectproof}. It
relies on both the estimate of the discrete Green's function outlined in Section~\ref{MGthmsectEJP} and the domain reduction argument given in Section~\ref{MGthmsectDR}. The purpose of the domain reduction argument is that it reduces the proof of Theorem~\ref{MGthm} to showing that~\eqref{MGthm.eq} holds for a special class of grid domains.

\begin{definition*}
We call a domain $D \subset \C$ a \emph{union of big squares (or UBS) domain} if $D$ can be written as
\[
D = \bigcup_{z\in V}\bs{z},
\]
where 
\[
\bs{z} = \{w\in\C:|\real(w)-\real(z)|<1, \, |\imag(w)-\imag(z)|<1\}
\] 
for some connected subset $V  \subset \Z^2$. 
\end{definition*}
Note that $\bs{z}$ is the \emph{open} square with side length $2$ around the vertex $z$. Furthermore, observe that a UBS domain is a grid domain, although the converse is not true. It will be tacitly understood that UBS domains are simply connected unless otherwise stated.

The main reason for using UBS domains is that while grid domains may have parts of the boundary with positive continuous harmonic measure but zero discrete harmonic measure, this does not happen with UBS domains if the discrete harmonic measure is interpreted appropriately. At the same time we can associate a UBS domain to each grid domain in $\grid$ without them differing too much from the conformal mapping point of view.

\subsection{Estimates of the discrete Green's function}\label{MGthmsectEJP}

The first step in the proof of Theorem~\ref{MGthm} requires the following estimate which is a version for UBS domains of Proposition~3.10 of~\cite{KozL1}.

\begin{proposition}\label{KLtheorem}
Let $0<\epsilon <1/4$ and let $0<\rho<1$ be fixed. There exists $R_0 < \infty$ such that the following holds. 
Suppose that $D$ is a UBS domain with $\inrad(D)=R$ and that $R> R_0$. Let  $V = V(D) = D \cap \Z^2$.
If $x$, $y \in V$ with $|\psi_D(x)| \le \rho$ and $|\psi_D(y)| \ge 1 - R^{-(1/4-\epsilon)}$,
then
\begin{equation}\label{KLestimate}
\frac{G_D(x,y)}{G_D(y)}= \frac{1-|\psi_D(x)|^2}
{|\psi_D(x)-e^{i\theta_D(y)}|^2} \cdot [\; 1 + O(R^{-(1/4-\epsilon)})  \;]
\end{equation}
where $G_D$ denotes the Green's function for simple random walk on $V$.
\end{proposition}

In~\cite{KozL1}, the results are proved for simply connected domains with a Jordan boundary and allow both points to be close to the boundary as long as they are not too close to each other. In the present paper, we are concerned with grid domains which, although still simply connected, need not have a Jordan boundary.  Furthermore, we are not concerned with any two arbitrary points, but rather with one point near the boundary and one point near the origin. Using this additional hypothesis and improving the methods of~\cite{KozL1} allows us to find a better exponent of $1/4$.

The derivation of Proposition~\ref{KLtheorem} in our particular setting essentially follows the same steps as in the original proof from~\cite{KozL1}. There is, however, the matter of adapting the original proof from a simply connected domain with Jordan boundary to a UBS domain.  This change of setting requires that certain technical estimates be established. For this reason we have included the proof of Proposition~\ref{KLtheorem} in this new setting in Appendix~\ref{prop42appendix}.

\subsection{A domain reduction}\label{MGthmsectDR}

Suppose that $D \in \grid$ is a grid domain and that $u \in \bd D \cap \Z^2$ is accessible by a simple random walk starting from $0$. Write $R=\inrad(D)$. Let $V = V(D)  = D \cap \Z^2$ denote those vertices contained in $D$ and let $V_0$ be the component of $V$ containing the origin; note that $V_0$ is simply
connected. Define $D_0 \subset D$ by setting
$$D_0 = \bigcup_{z\in V_0}\bs{z},$$
where $\bs{z} = \{w\in\C:|\real(w)-\real(z)|<1,\, |\imag(w)-\imag(z)|<1\}$ 
so that $D_0$ is a UBS domain. We will call $D_0$ the UBS domain \emph{associated} with $D$. In particular, notice that
\begin{itemize}
\item[{\rm (i)}]  $D_0 \subset D$ is a simply connected domain containing the origin,
\item[{\rm (ii)}] $u \in \partial D_0$, and
\item[{\rm (iii)}] for some $1 \le M \le \infty$, we can write 
\[
\partial D_0 \cap D= \bigcup_{j=1}^M \mathcal{C}_j,
\]
where $\mathcal{C}_j$, $j=1,\ldots, M$, are crosscuts of $D$ with length at most $2$.
\end{itemize}

For ease of notation, throughout this section, we write  $\psi$ for
$\psi_D$ and $\psi_0$ for $\psi_{D_0}$. Recall that we can write
$\psi(z) = \exp\{-g(z) + i \theta(z)\}$
 and
$\psi_0(z) = \exp\{-g_{0}(z) + i \theta_{0}(z)\}$
where $g$ and $g_{0}$ are the Green's functions for $D$ and $D_0$, respectively.
 
 By Lemma~\ref{DRbeurling}, since $\diam(\mathcal{C}_j) \le
2$, there exists a universal constant $c < \iy$ such that
\[
\diam(\psi(\mathcal{C}_j)) \le c R^{-1/2}.
\]
If $\Omega=\psi(D_0) \subset \D$ it follows that
\[
\partial \Omega \cap \D \subset \mathcal{A}(1-cR^{-1/2},1),
\]
where $\mathcal{A}(a,b) = \{z: a < |z| < b\}$.
Finally, we write 
 \begin{equation}\label{defnvp}
\psi_0 = \vp \circ \psi, \quad z \in D_0,
\end{equation}
where $\vp: \Omega \to \D$ is the conformal map of $\Omega=\psi(D_0)$ onto $\D$ satisfying $\vp(0)=0$, $\vp'(0)>0$. The following estimate quantifies the fact that $\vp$ is almost the identity away from the boundary.

\begin{lemma}\label{ring.lemma} Let $0 < \ee < 1/2$ be fixed.
Suppose $\Omega \subset \D$ with $\partial \Omega \cap \D \subset \mathcal{A}(1-\ee,1)$, and let
$\vp: \Omega \to \D$ be the conformal map of $\Omega$ onto $\D$  with $\vp(0)=0$,  $\vp'(0)>0$. If $|z| \le 1-2\ee$, then
\[
|\vp(z)-z| \le c_0 \ee \, \log (1/\ee),
\] 
where $c_0$ is a uniform constant.
\end{lemma}

\begin{proof}
In Section~3.5 of~\cite{SLEbook} it is shown that there is a universal constant $c_1$ such that
\[
|\log(\vp(z)/z)| \le c_1 \ee [1-\log(1-|z|)],
\]
where the branch of the logarithm is chosen so that
$\log(\vp(0)/0)=\log \vp'(0) \ge 0$. It follows that if $|z|\le 1-2\ee$, then using the Schwarz lemma and a Taylor expansion, there is a constant $c_2$ depending only on $c_1$ such that 
\[
0 \le |\vp(z)|-|z| \le c_2 \ee \, \log (1/\ee),
\]
and 
\begin{equation} \label{nov10.1}
|\arg (\vp(z)/z)| \le c_2 \ee [ 1+\log (1/\ee)],
\end{equation}
completing the proof.
\end{proof}

\begin{lemma} \label{kozdron-lawler-2-lemma}
Let $0<\epsilon<1/4$ be fixed. There exists $R_0$ such that if $R > R_0$ and $x$, $y \in V_0$ with 
\begin{equation} \label{april14.1}
g(x) \ge R^{-(1/4-\epsilon)}
\end{equation}
and $g(y) < R^{-(1/4-\epsilon)}$, then
\begin{equation}\label{aug3.eq2}
\psi_0(x) = \psi(x) + O(R^{-1/2}\log R)
\end{equation}
and
\begin{equation}\label{aug3.eq3}
e^{i\theta_{0}(y)}=e^{i\theta(y)} + O(R^{-1/4}).
\end{equation}
\end{lemma}

\begin{proof}
We may assume that $x \neq 0$. Let $c_0>0$ be a constant such that $\partial \psi(D_0) \subset \mathcal{A}(1-c_0R^{-1/2},1)$ and recall that $\psi_0=\vp \circ \psi$ for $z \in D_0$ as in~\eqref{defnvp}.
Note that~\eqref{april14.1} implies that $|\psi(x)| \le 1 - cR^{-(1/4-\epsilon)} \le 1-2c_0R^{-1/2}$, for $R$ large enough, so
Lemma~\ref{ring.lemma} applied to the point $z=\psi(x)$ implies that there exists a uniform constant $c_1$ such that
\begin{equation*}
|\vp(\psi(x))-\psi(x)|  = |\psi_0(x) - \psi(x)| \le c_1 R^{-1/2} \, \log R ,
\end{equation*}
yielding~\eqref{aug3.eq2}. 

If $y$ is as in the statement of the lemma and $|\psi(y)| \le 1-2c_0R^{-1/2}$ then~\eqref{aug3.eq3} follows from~\eqref{nov10.1}. Hence we may assume that $|\psi(y)| > 1-2c_0R^{-1/2}$.
Since the boundary of $\psi(D_0)$ contained in $\D$ is a union of images of crosscuts with diameter bounded by $c_0 R^{-1/2}$ there is a curve $\beta$ in $\psi(D_0)$ that connects $\psi(y)$ to the circle $\{|z| = 1-2c_0R^{-1/2}\}$
and satisfies $\diam  \beta \le c_2 R^{-1/2}$ for some absolute constant $c_2 < \infty$. By Lemma~\ref{DRbeurling}
we have $\diam \vp(\beta) \le c_3R^{-1/4}$ and using again Lemma~\ref{ring.lemma} we see that 
\begin{equation*}
e^{i\theta_{0}(y)}=e^{i\theta(y)} + O(R^{-1/4})
\end{equation*}
yielding~\eqref{aug3.eq3}, 
and the proof is complete.
\end{proof}

The final result for this section uses a particular continuity estimate for the Poisson kernel which we now state as a separate lemma.

\begin{lemma}
 If  $z \in \Disk$, $w \in \bd \Disk$,  and
\begin{equation}\label{PKformDR}
\PK{z}{w}{\Disk} = \frac{1-|z|^2}{|z-w|^2}
\end{equation}
so that $\PKa/(2\pi)$ is the Poisson kernel for the unit disk, then
 \begin{align}\label{PKrate.eq}
 |\PK{z'}{w'}{\Disk} &- \PK{z}{w}{\Disk} |  \nonumber\\
&\le |z'-z| \biggl[ \frac{  4|z-w| + 4|z'-w|}{|z-w|^2|z'-w|^2}\nonumber\\
&\qquad\qquad\quad\quad +  \frac{  (|z-w| + |z'-w|+2)(|z-w|^2 +|z'-w|^2)}{|z-w|^2|z'-w|^2}  \biggr]\nonumber\\
&\qquad+ |w-w'| \left[\frac{ |z'-w'| + 3|z'-w|}{|z'-w|^2|z'-w'|^2}\right].
\end{align}
\end{lemma}

\begin{proof} 
In order to derive this estimate for the Poisson kernel, we will write
$$|\PK{z'}{w'}{\Disk}  - \PK{z}{w}{\Disk}  |
\le 
|\PK{z'}{w}{\Disk}  - \PK{z}{w}{\Disk}  | +
|\PK{z'}{w'}{\Disk}  - \PK{z'}{w}{\Disk}  |$$
and estimate each piece separately. For the first piece, we begin by noting that  
\begin{equation}\label{dec28eq2}
|z'-w|^2 - |z-w|^2 \le |z'-z|^2 +2|z'-z||z-w|.
\end{equation}
Switching $z$ and $z'$ in~\eqref{dec28eq2} gives
\begin{equation}\label{dec28eq2b}
|z-w|^2 - |z'-w|^2 \le |z'-z|^2 +2|z'-z||z'-w|.
\end{equation}
Furthermore, 
\begin{align}\label{dec28eq4}
|z'|^2&|z-w|^2 - |z|^2|z'-w|^2 \\ \notag
&\le |z-z'|^2|z-w|^2 +2|z||z-z'||z-w|^2 +|z|^2(|z-w|^2 - |z'-w|^2)\\ \notag
&\le |z-z'|^2|z-w|^2 +2|z-z'||z-w|^2 + |z'-z|^2 +2|z'-z||z'-w|
\end{align}
using~(\ref{dec28eq2b}) the fact that  $|z| \le 1$. It now follows that
\begin{align*}
&\PK{z}{w}{\Disk}  - \PK{z'}{w}{\Disk}
=\frac{|z'-w|^2 - |z-w|^2   +|z'|^2|z-w|^2 -|z|^2|z'-w|^2 }{|z-w|^2|z'-w|^2}\\
&\le |z'-z| \left[\frac{  2|z'-z| +2|z-w|+ |z-z'||z-w|^2 +2|z-w|^2  +2|z'-w| }{|z-w|^2|z'-w|^2}\right]
\end{align*}
 using~(\ref{dec28eq2}) and~(\ref{dec28eq4}).  If we now use the fact that $|z-z'| \le |z-w| + |z'-w|$, then
\begin{align}\label{newfeb1eq1}
&\PK{z}{w}{\Disk}  - \PK{z'}{w}{\Disk} \notag \\
&\le |z'-z| \left[\frac{  4|z-w| + 4|z'-w|+ (|z-w| + |z'-w|+2)|z-w|^2}{|z-w|^2|z'-w|^2}\right].
\end{align}
Switching $z$ and $z'$ in~\eqref{newfeb1eq1} implies that 
\begin{align*}
&|\PK{z'}{w}{\Disk}  - \PK{z}{w}{\Disk} |\\
& \le\! |z'-z| \!\!\left[\frac{  4|z-w| + 4|z'-w|+  (|z-w| + |z'-w|+2)(|z-w|^2 +|z'-w|^2)}{|z-w|^2|z'-w|^2} \right]\!\!.
\end{align*}
As for the second piece, note that 
\begin{align*}
| \PK{z'}{w'}{\Disk}-\PK{z'}{w}{\Disk}  |
&\le \left|\frac{|z'-w'|^2 - |z'-w|^2}{|z'-w|^2|z'-w'|^2} \right|\\
&\le \left|\frac{(|z'-w| + |w-w'|)^2 - |z'-w|^2}{|z'-w|^2|z'-w'|^2} \right|\\
&=|w-w'| \left[\frac{|w-w'| +2|z'-w|}{|z'-w|^2|z'-w'|^2}\right].
\end{align*}
Here we are using the fact that
\begin{equation}
 ||z'-w'| ^2 - |z'-w|^2 | \le  | (|z'-w| +|w-w'|) ^2 - |z'-w|^2 |
 \end{equation}
which follows by considering separately the two possible cases $|z'-w| \le  |z'-w'|$ and $|z'-w'| \le  |z'-w|$.
If we now use the fact that $|w-w'| \le |z'-w| + |z'-w'|$, then
$$ |\PK{z}{w}{\Disk}  - \PK{z}{w'}{\Disk} |
\le |w-w'| \left[\frac{|z'-w'| + 3|z'-w|}{|z'-w|^2|z'-w'|^2}\right].$$
Thus, combining both our estimates gives the required result.
\end{proof}

\begin{lemma} \label{kozdron-lawler-2-lemma2}
Let $0<\epsilon <1/4$ and let $0<\rho<1$ be fixed. There exists $R_0 < \infty$ such that if $R > R_0$ and if $x$, $y \in V_0$ with $|\psi(x)| \le \rho$ and $|\psi(y)| \ge 1 - R^{-(1/4-\epsilon)}$,
then
\begin{equation}\label{aug3.eq4}
\frac{1-|\psi_0(x)|^2}{|\psi_0(x)-e^{i\theta_0(y)}|^2}
= \frac{1-|\psi(x)|^2}{|\psi(x)-e^{i\theta(y)}|^2} + O(R^{-1/4}).
\end{equation}
\end{lemma}

\begin{proof}
Let $z=\psi(x)$,  $z'= \psi_0(x)$, $w=e^{i\theta(y)}$, and $w'=e^{i\theta_{0}(y)}$, and note that by assumption there exists some constant $0<\rho'<1$ such that $|z-w| \ge \rho'$.
We also know from Lemma~\ref{kozdron-lawler-2-lemma} that there exist constants $c_1$ and $c_2$ such that
\begin{equation*}
|w - w'| \le c_1 R^{-1/4} \;\;\; \text{and} \;\;\; |z-z'| \le c_2 R^{-1/2} \log R.
\end{equation*}
Using the crude bounds that $|z-w| \le 2$, $|z'-w| \le 2$, $|z-w'| \le 2$, and $|z'-w'| \le 2$, it follows from~\eqref{PKrate.eq} that for $R$ large enough
\begin{equation}\label{PKbound.eq1}
 |\PK{z'}{w'}{\Disk}  - \PK{z}{w}{\Disk}  |  
\le c_3R^{-1/4}.
\end{equation}
Thus, we see from~\eqref{PKformDR} that~\eqref{PKbound.eq1} is equivalent to~\eqref{aug3.eq4} as required, and the proof is complete.
\end{proof}

\subsection{Proof of Theorem~\ref{MGthm}}\label{MGthmsectproof}

Let $D \in \grid$ be a grid domain, write $R=\inrad(D)$, and assume that $u \in V_\bd(D)$ is accessible by a simple random walk starting from $0$.  Let $V = V(D)  = D \cap \Z^2$, let $V_0$ be the component of $V$ containing the origin, and let $D_0$ be the UBS domain associated to $D$ as in Section~\ref{MGthmsectDR}.
Recall that $D_0 \subset D$ is a simply connected domain containing the origin and $u \in V_\partial(D_0)$. 

As in~\eqref{GFdecomp}, if $z \in V_0$ and $w \in V_{\bd}(D_0)$, then
\begin{equation}\label{step1}
H_{D_0}(z,w) = \frac{1}{4} \sum_{A_{w}} G_{D_0}(z,y)
\end{equation}
where $A_{w}$ is as in Section~\ref{conformal}.

Recall that we can write $\psi_{D_0}=\vp \circ \psi_D$. Hence if $|\psi_{D}(x)| \le \rho$, there is a $\rho_0<1$ only depending on $\rho$ such that $|\psi_{D_0}(x)| \le \rho_0$ whenever $R$ is sufficiently large. Since $D_0$ is a UBS domain, we can apply Proposition~\ref{KLtheorem} to $u$ and any point $x \in V_0$ with $|\psi_{D_0}(x)|\le \rho_0$. Hence, substituting~\eqref{KLestimate} into~\eqref{step1} gives
\begin{equation}\label{step2}
H_{D_0}(x,u) = \frac{1}{4} \sum_{A_{u}} G_{D_0}(y)\cdot \frac{1-|\psi_{D_0}(x)|^2}
{|\psi_{D_0}(x)-e^{i\theta_{D_0}(y)}|^2} \cdot [\; 1 + O(R^{-(1/4-\epsilon)})  \;].
\end{equation}
Since the summation in~\eqref{step2} is over $y$, we use the fact~\eqref{thetaestimate} that
$$\psi_{D_0}(u) = e^{i\theta_{D_0}(y)} + O(R^{-1/2})$$
to conclude
\begin{equation}\label{step3}
H_{D_0}(x,u) =  \frac{1-|\psi_{D_0}(x)|^2}
{|\psi_{D_0}(x)-\psi_{D_0}(u) |^2} \cdot [\; 1 + O(R^{-(1/4-\epsilon)})  \;]
\cdot \frac{1}{4} \sum_{A_{u}} G_{D_0}(y).
\end{equation}
Since 
$$H_{D_0}(0,u) =  \frac{1}{4} \sum_{A_{u}} G_{D_0}(y),$$
we see that~\eqref{step3} yields
\begin{equation*}
\frac{H_{D_0}(x,u)}{H_{D_0}(0,u)} = \frac{1-|\psi_{D_0}(x)|^2}
{|\psi_{D_0}(x)-\psi_{D_0}(u) |^2} \cdot [\; 1 + O(R^{-(1/4-\epsilon)})  \;].
\end{equation*}
If we now observe that
\begin{equation}\label{step5}
\frac{H_{D}(x,u)}{H_{D}(0,u)}  = \frac{H_{D_0}(x,u)}{H_{D_0}(0,u)} 
\end{equation}
since $V_0$ consists of precisely those vertices accessible by a  simple random walk starting from the origin, and that Lemma~\ref{kozdron-lawler-2-lemma2} combined with~\eqref{thetaestimate} implies
\begin{equation}\label{step6}
\frac{H_{D_0}(x,u)}{H_{D_0}(0,u)} = \frac{1-|\psi_{D}(x)|^2}
{|\psi_{D}(x)-\psi_{D}(u) |^2} \cdot [\; 1 + O(R^{-(1/4-\epsilon)})  \;],
\end{equation}
then combining~\eqref{step5} and~\eqref{step6} gives~\eqref{MGthm.eq} and the proof of Theorem~\ref{MGthm} is complete.

\section{Moment estimates for increments of the driving function}\label{keysect}

The idea is now to use Theorem~\ref{MGthm} to transfer the fact that a suitable version of the discrete Poisson kernel~\eqref{MGthm.eq} is a martingale with respect to the growing loop-erased random walk path to information about the Loewner driving function for a mesoscopic scale piece of the path. This is the analogue of Proposition~3.4 of~\cite{LSW-aop}, but with a rate of decay. 
Suppose that $D \in \grid$ is a grid domain, write $R=\inrad(D)$, and let $\psi_D: D \to \D$ be the conformal map of $D$ onto $\D$ with $\psi_D(0)=0$, $\psi_D'(0)>0$. For ease of notation, we will write $\psi=\psi_D$ in what follows. 
For $w \in D$ and $u \in \bd D$, define
\begin{equation*}
\PK{w}{u}{D}=\text{Re}\left(\frac{\psi(u)+\psi(w)}{\psi(u)-\psi(w)}\right)= \frac{1-|\psi(w)|^2}{|\psi(w)-\psi(u)|^2}
\end{equation*}
as in~\eqref{PKformDR}.

Let $\gamma=(\gam_0,\ldots, \gam_l)$ denote the loop-erasure of the
time-reversal of simple random walk started at $0$, stopped when it
hits $\bd D$, and for $j\geq 0$, define the slit domains
\begin{equation*}
D_j=D\setminus\bigcup_{i=1}^j[\gam(i-1),\gam(i)].
\end{equation*}
As before, the conformal maps $\psi_j:D_j\to\Disk$ will be those satisfying $\psi_j(0)=0$ and $\psi_j'(0)>0$. 
We write $t_j$ for the capacity of the curve $\psi(\gam[0,j])$ from $0$ in $\Disk$. Denote by $W:[0,\infty)\to \bd\Disk$ the Loewner driving function for the curve $\tilde\gamma^R = \psi(\gamma)$ parameterized by capacity. That is, $W$ is the unique continuous function such that solving the radial Loewner equation~\eqref{LODE} with driving function $W$ gives the path $\tilde \gam^R$. 
Moreover, we denote by $(\vartheta(t),t\geq 0)$ the continuous, real-valued function such that $\vartheta(0)=0$ and 
\[
W(t)=W(0)e^{i\vartheta(t)},
\] and we define 
\[
\Delta_j=\vartheta(t_j).
\] 
Let $0<\epsilon<1/4$ be fixed. Set  $3\alpha =1/4-\epsilon$ and define  
\begin{equation}\label{m}
m=m(R)=\min\{j\geq 0: t_j\geq R^{-2\alpha} \text{ or } |\Delta_j| \geq R^{-\alpha}\}.
\end{equation}
The following is Lemma~2.1 of~\cite{LSW-aop}.

\begin{lemma}\label{lemma21}
Suppose $K_t$ is the hull obtained by solving~\eqref{LODE} with
$U_t$ as driving function. If $D(t)=\sqrt{t} + \sup_{0 \le s \le t}\{|U_s-U_0|\}$, 
then there exists a constant $c$ such that
\begin{equation*}
c^{-1}\min\{1, D(t)\} \le \diam(K_t) \le c D(t).
\end{equation*}
\end{lemma}

It follows from the last lemma and Lemma~\ref{DRbeurling} that
$t_m \le R^{-2\alpha} + O(R^{-1})$ and 
\begin{equation}\label{Deltaupperbound}
|\Delta_m| \le
R^{-\alpha} + O(R^{-1/2}). 
\end{equation}
Furthermore, if $w=\psi(v)$ where $|v| \le \inrad(D)/5$, then the Koebe one-quarter theorem implies $|w| \le 4/5$. 
By the Loewner equation, we have
\begin{equation*}
|\psi_{m}(v)-w| \le c R^{-2\alpha} 
\end{equation*}
so that $|\psi_m(v)|\le 5/6$ for $R$ large enough. This means that the conditions of Theorem~\ref{MGthm} are satisfied by $v \in D_j$ for each $1 \le j \le m$.

Let $x \in D \cap \Z^2$, $w \in V_{\bd}(D)$, and recall the definition of the hitting probability 
$H_D(x,w)$ from Section~\ref{greenZ}. We will write $H_j(x,w)$ for $H_{D_j}(x,w)$. 
Fix $v\in V(D)$ with $|v|\leq R/5$. 
It can be shown that 
\begin{equation*}
M_j= \frac{H_j(v,\gam_j)}{H_j(0,\gam_j)}
\end{equation*}
is a martingale with respect to the filtration generated by
$\gamma[0,j]$, $j \ge 0$; see~\cite{LSW-aop}.
With the definition $\lambda_j =\lambda(v,\gam_j;D_j)$, we know from
Theorem~\ref{MGthm} that 
$$\left|\frac{H_j(v,\gam_j)}{H_j(0,\gam_j)}-\lambda_j\right|\leq cR^{-3\alpha}$$
for $j\leq m$ implying that
\begin{equation*}
\E[\lambda_m-\lambda_0] = \E[M_m-M_0]+O(R^{-3\alpha}) = O(R^{-3\alpha}).
\end{equation*}
By a Taylor expansion using the Loewner equation we get 
\[
\lambda_m-\lambda_0=\real\left(\frac{ZU(U+Z)}{(U-Z)^3}\right)(2t_m-\Delta_m^2)+2\imag\left(\frac{ZU}{(U-Z)^2}\right)\Delta_m+O(R^{-3\alpha}),
\]
where $Z=\psi(v)$ and $U=W(0)$. (See~Remark~3.6 and the proof of Proposition~3.4 in~\cite{LSW-aop} for more details.) By taking the expectation and plugging in two different $v$, exactly as in~\cite{LSW-aop},
 recalling that $3\alpha= 1/4-\epsilon$, we arrive at the following. 

\begin{proposition}\label{keyestimate}
Let $0<\epsilon<1/4$ be fixed. There exist
constants $c>0$, $R_0\geq 1$ such that for all $R\geq R_0$ the following
holds. Let $D \in \grid$ be a grid domain with $\inrad(D)=R$ and let $\gamma$ be the loop erasure  of the time-reversal of simple random walk from $0$ in $D$ conditioned to
  exit $D$ through an edge corresponding to $u_0$, where  $u_0 \in V_{\bd}(D)$ is such that
this event has positive probability. If $t_j$, $\Delta_j$, and $m$ are defined as above, then
\begin{equation*}
|\E[\Delta_m]|\leq cR^{-(1/4-\epsilon)}
\end{equation*}
and
\begin{equation*}
|\E[\Delta_m^2]-2\E[t_m]|\leq cR^{-(1/4-\epsilon)}.
\end{equation*}
\end{proposition}

\section{Skorokhod embedding and proof of Theorem~\ref{the.theorem2}}\label{Sect-theproof}

Assume that $D \in \grid$ and write $R=\inrad(D)$.
Recall that the Loewner driving function for the loop-erased random walk path $\tilde\gamma^R=\psi_D(\gamma)$ in $\D$ is denoted $W(t)=W_0e^{i \vt(t)}$. 
In Proposition~\ref{HM_ROC} we quantified that $W_0$ is close to uniform in terms of the inner radius $R$. Hence, to prove Theorem~\ref{the.theorem2} it will be enough to study $\vt(t)$, and show that it is close to a standard Brownian motion with speed 2. One way of proving this is to couple (a variant of) this process with Brownian motion, using Skorokhod embedding. The standard version of this technique is a method for coupling
sums of i.i.d.\ random variables and Brownian motion in such a way that with large probability the processes are close at any given time. In the proof, a sequence of times $\{t_{m_k}\}_{k\geq 1}$ is constructed which correspond to roughly constant increases in capacity for the time-reversed loop-erased random walk in $D$. Although $\{\vt(t_{m_k})\}_{k\geq 1}$ is not a random walk, it is almost a martingale, and in view of Section~\ref{keysect} we can use  the following version of Skorokhod embedding for martingales. 
Both~(6.1) and~(6.2) come directly from Theorem~A.1 in~\cite{Hall} while (6.3) follows from the proof.

\begin{lemma}[Skorokhod embedding theorem]
\label{skor}
Suppose $(M_k)_{k \le K}$ is an $({\mathcal F}_k)_{k \le K}$ martingale,
with  $\| M_{k+1}-M_{k}\|_\infty\le \delta$ and $M_0=0$ a.s.
There are stopping times
$0=\tau_0\le\tau_1\le \cdots \le\tau_K$  for standard Brownian motion 
$B(t)$, $t \ge 0$, such that
$(M_0,M_1, \ldots, M_K)$ and $(B(\tau_0),B(\tau_1), \ldots, B(\tau_K))$ 
have the same law.
Moreover, we have for $k=0,1,\dots,K-1$, 
\begin{align}\label{e.etau}
\E \bigl[ \tau_{k+1} - \tau_k  \,|\, B[0, \tau_k] \bigr]
&= \E \bigl[ (B(\tau_{k+1}) - B(\tau_k) )^2 \,|\, B [0, \tau_k]
\bigr], \\
\E \bigl[ (\tau_{k+1} - \tau_k)^p  \,|\, B[0, \tau_k] \bigr]
&\le C_p \E \bigl[ (B(\tau_{k+1}) - B(\tau_k))^{2p} \,|\, B [0,
\tau_k] 
\bigr] \label{e.moment},
\end{align}
for constants $C_p < \infty$,
and also
\begin {equation}
\label {e.bdt}
\tau_{k+1} \le \inf\left\{ t \ge \tau_k  : \, | B(t) - B(\tau_k)| \ge 
 \delta \right\}.
\end {equation}
\end{lemma}

We will now prove Theorem~\ref{the.theorem2} using Proposition~\ref{keyestimate} and Lemma~\ref{skor}. Although the structure of the proof is similar to that of Theorem~3.7 in~\cite{LSW-aop}, some estimates need to be done with more care, in particular to ensure that the exponent in our rate of convergence is optimal for the method used in this paper. Rather than including the key steps and referring the reader to~\cite{LSW-aop}, we write the proof in detail here to allow a more fluid reading. 

The following result about the modulus of continuity of Brownian motion will be needed; see Lemma~1.2.1 of~\cite{Csorgo} for the proof.

\begin{lemma}\label{moc}
Let $B(t)$, $t \ge 0$, be standard Brownian motion. For each $\ee
>0$ there exists a constant $C=C(\ee)>0$ such that the inequality
\[
\PP\left(\sup_{t \in[0, T-h]}
  \sup_{s \in (0, h]} |B(t+s)-B(t)| \le v \sqrt{h}\right) \ge 1-
  \frac{CT}{h} e^{-\frac{v^2}{2+\ee}}
\]
holds for every positive $v$, $T$, and $0 < h < T$.
\end{lemma}

The proof of convergence in~\cite{LSW-aop} uses Doob's maximal inequality. In order to obtain a better rate of convergence, we need a sharper maximal inequality for martingales, namely Lemma~1 of~\cite{haeusler}.

\begin{lemma}\label{haeusler}
Let $\xi_k, \, k=1,\ldots,K$, be a martingale difference sequence with
respect to the filtration $\mathcal{F}_k$.
If
$\lambda, u, v >0$, then it follows that
\begin{align*}
\PP\left(\max_{1\le j \le K}|\sum_{k=1}^j \xi_k| \ge \lambda\right) \le 
&\sum_{k=1}^K\PP(|\xi_k|>u)  \\
&\qquad+2\PP\left(\sum_{k=1}^K\E[\xi_k^2|\mathcal{F}_{k-1}] > v\right)\\
&\qquad\qquad+ 2\exp\{\lambda
u^{-1}(1-\log(\lambda u v^{-1}))\}.
\end{align*}
\end{lemma}

The strategy of the proof of Theorem~\ref{the.theorem2} is the following. In Section~\ref{keysect}, we showed that $\E[\Delta_m]$ is close to zero. We use the domain Markov property to iterate this estimate to construct a sequence of random variables $\Delta_{m_k}$ that almost forms a martingale. We adjust 
the sequence $\Delta_{m_k}$ to make it into a martingale, so that we can couple it with Brownian motion, using Skorokhod embedding.

The next step is to show that the stopping times $\tau_k$ obtained by Skorokhod embedding are likely to be close to the capacities $2t_{m_k}$ for all $k\leq K$ for some appropriate $K$. This is done by showing separately that each of these two quantities has high probability of being close to the natural time (the quadratic variation) of the martingale.

Once that we know that the two processes run on similar clocks, all that is left to do is show that they are likely to be close at all times. The key tool needed for that is Lemma~\ref{moc}.

\begin{proof}[Proof of Theorem~\ref{the.theorem2}] Choose, without loss of generality, $T \ge 1$ and assume $R \ge R_1> 8e^{20T}R_0$,
where $R_0$ is the constant from Proposition~\ref{keyestimate}. 
This choice of $R_1$ implies that Proposition~\ref{keyestimate} can be applied to $D$ slit by the initial piece of curve $\gamma$ up to capacity $20T$. Indeed, the Koebe one-quarter theorem implies that $\inrad(D \setminus \beta)/\inrad(D) \ge \exp\{-\ecap(\beta)\}/4$ if $D$ is slit by the curve $\beta$.

In what follows, most constants, which may depend on $T$, will be denoted by $c$ even though they may change from one line to the next. Define $m_0=0$ and $m_1=m$, where $m$ is defined as in~\eqref{m}. Inductively for $k=1, 2, 3, \ldots,$ define 
\[
m_{k+1}=\min\{ j > m_k: |t_j-t_{m_k}|\ge R^{-2\alpha} \text{ or } |\Delta_j-\Delta_{m_k}| \ge R^{-\alpha}\}.
\]
Define
\begin{equation}\label{K}
K= \lceil 10TR^{2\alpha}\rceil
\end{equation}
and note that $t_{m_K} \le 20 T$. Set $\eta(R)=R^{-\alpha}$, where $\alpha = (1/4-\eps)/3$. Then, by Proposition~\ref{keyestimate} and the domain Markov property of loop-erased random walk, we can find a universal constant $c$ such that
\begin{equation}\label{DeltaInc}
|\E[\Delta_{m_{k+1}}-\Delta_{m_k}|
\mathcal{F}_k]| \le c\eta^3
\end{equation}
and 
\begin{equation}\label{DeltaVar}
|\E[(\Delta_{m_{k+1}}-\Delta_{m_{k}})^2-2(t_{m_{k+1}}-t_{m_{k}})|\mathcal{F}_k]|
\le c \eta^3,
\end{equation}
for $k=0, \ldots, K-1$, where $\F_k$ is the filtration generated by $\gamma_n[0,m_k]$.

For $j=1,\ldots, K$, define
\[
\xi_j=\Delta_{m_{j}}-\Delta_{m_{j-1}}-\E[\Delta_{m_{j}}-\Delta_{m_{j-1}}|\mathcal{F}_{j-1}].
\] 
This is clearly a martingale difference sequence and $M$ defined by $M_0=0$ and 
$$M_k=\sum_{j=1}^k \xi_j$$
for $k=1,\ldots, K,$ is a martingale with respect to
$\F_k$. Note that 
$$\| M_{k}-M_{k-1}\|_{\infty}\le 4\eta$$
 by~\eqref{Deltaupperbound} for $R$ sufficiently large.

Skorokhod embedding allows us to find stopping times
$\{\tau_k\}$ for standard Brownian
motion $B$ and a coupling of $B$ with the martingale $M$ (and the loop-erased random walk path $\gamma$) such that $M_k=B(\tau_k)$, 
$k=0, \ldots, K$. 

Consider the natural time associated to $M$, namely
\[Y_k=\sum_{j=1}^k \xi_j^2, \quad k=1, \ldots, K.\]
We will show that  $2t_{m_k}$ is close to the
stopping time $\tau_k$ for every $k\leq K$ by showing separately that each of these quantities is close to $Y_k$.
We first show that  $Y_k$ is close to $2t_{m_k}$ for every $k\leq K$. Set $\sigma_k=2t_{m_k}-2t_{m_{k-1}}$. For $\phi=3\eta|\log \eta|$ we have
\begin{align*}
\PP&\left(\max_{1 \le k \le K}|\sum_{j=1}^k(\xi_j^2 - \sigma_j)| \ge
  \phi\right)\\
&\qquad\qquad\le \PP\left(\max_{1 \le k \le
    K}|\sum_{j=1}^k(\xi_j^2 - \E[\xi_j^2|\mathcal{F}_{j-1}])| \ge
  \phi/3\right) \\
&\qquad\qquad\qquad+ \PP\left(\max_{1 \le k \le
    K}|\sum_{j=1}^k(\E[\xi_j^2|\mathcal{F}_{j-1}]-\E[\sigma_j|\mathcal{F}_{j-1}])| \ge
  \phi/3\right) \\
&\qquad\qquad\qquad\qquad+ \PP\left(\max_{1 \le k \le
    K}|\sum_{j=1}^k(\sigma_j-\E[\sigma_j|\mathcal{F}_{j-1}])| \ge
  \phi/3\right) \\
&\qquad\qquad=:p_1+p_2+p_3.
\end{align*}
We estimate $p_1$ using the maximal inequality from Lemma~\ref{haeusler} with $\lambda=\eta|\log \eta|$, $u=\eta$, and
$v=e^{-2}\lambda u$. This gives
\begin{align*}
p_1 &\le  \sum_{j=1}^K \PP\left( |\xi_j^2-\E[\xi_j^2|\mathcal{F}_{j-1}]|
  > \eta \right)\\
&\qquad + 2\PP\left( \sum_{j=1}^K
  \E\left[(\xi_j^2-\E[\xi_j^2|\mathcal{F}_{j-1}])^2| \F_{j-1}\right]
  > e^{-2}\eta^2|\log \eta|
\right)\\
&\qquad\qquad+2\eta.
\end{align*}
Since $\max_j|\xi_j|\le 4\eta$, the first sum is equal to zero for $R$ sufficiently large. This bound
and the definition of $K$ imply that 
\[\sum_{j=1}^K
  \E\left[(\xi_j^2-\E[\xi_j^2|\mathcal{F}_{j-1}])^2| \F_{j-1}\right]
  \le 16\lceil 10 T \rceil \eta^2.
\] 
It follows that the second sum also is zero if $R$ is large enough. To get a bound on $p_2$ we note that by~\eqref{DeltaInc} and~\eqref{DeltaVar},
\begin{align*}
|\E[\xi_j^2|&\mathcal{F}_{j-1}]-\E[\sigma_j|\mathcal{F}_{j-1}]|\\
&=|\E[(\Delta_{m_j}-\Delta_{m_{j-1}})^2|\F_{j-1}]-2\E[t_{m_j}-t_{m_{j-1}}|\F_{j-1}]+
O(\eta^4)|\\
&\le c\eta^3.
\end{align*}
Using the triangle inequality and summing over $j$ we see that $p_2=0$ if
$R$ is large enough. Finally $p_3$ is estimated in a similar fashion
as $p_1$ using the inequality $\max_k \sigma_k \le 3\eta^{2}$.
This shows that
\begin{equation}\label{mg-time}
\PP\left(\max_{1 \le k \le K}|Y_k-2t_{m_k}| \ge 3\eta|\log \eta |\right) =O(\eta)
\end{equation}
for all $R$ large enough.

We now show that $Y_k$ is close to $\tau_k$ for every $k\leq K$.
Set $\zeta_k=\tau_{k}-\tau_{k-1}$ and let $\mathcal{G}_k$ denote the $\sigma$-algebra
generated by $B[0,\tau_k]$. Then, again with
$\phi=3\eta|\log \eta|$, we can write
\begin{align*}
\PP&\left(\max_{1 \le k \le K}|\sum_{j=1}^k(\xi_j^2 - \zeta_j)| \ge
  \phi\right)\\
&\qquad\qquad\le \PP\left(\max_{1 \le k \le
    K}|\sum_{j=1}^k(\xi_j^2 - \E[\xi_j^2| \mathcal{G}_{j-1}])| \ge
  \phi/3\right) \\
&\qquad\qquad\qquad+ \PP\left(\max_{1 \le k \le
    K}|\sum_{j=1}^k(\E[\xi_j^2|\mathcal{G}_{j-1}]-\E[\zeta_j|\mathcal{G}_{j-1}])| \ge
  \phi/3\right) \\
&\qquad\qquad\qquad\qquad+ \PP\left(\max_{1 \le k \le
    K}|\sum_{j=1}^k(\zeta_j-\E[\zeta_j|\mathcal{G}_{j-1}])| \ge
  \phi/3\right) \\
&\qquad\qquad=:p_4+p_5+p_6.
\end{align*}
The estimate of $p_4$ is identical to the estimate of $p_1$ above, and by~\eqref{e.etau} we conclude $p_5=0$. (Recall that $\xi_j^2=(B(\tau_j)-B(\tau_{j-1}))^2$.) It remains to
estimate $p_6$. We use Lemma~\ref{haeusler} to get
\begin{align}\label{p6}
p_6&\le  \sum_{j=1}^K \PP\left( |\zeta_j-\E[\zeta_j|\mathcal{G}_{j-1}]|
  > \eta \right) \nonumber\\
&\qquad + 2\PP\left( \sum_{j=1}^K
  \E\left[(\zeta_j-\E[\zeta_j|\mathcal{G}_{j-1}])^2| \mathcal{G}_{j-1}\right]
  > e^{-2}\eta^2|\log \eta|
\right)\\
&\qquad\qquad+2\eta \nonumber.
\end{align}
By the definition of $K$ as in~\eqref{K}, Chebyshev's inequality,~\eqref{e.etau} and \eqref{e.moment} we have 
\begin{align*}
\sum_{j=1}^K\PP\left( |\zeta_j-\E[\zeta_j|\mathcal{G}_{j-1}]|
  > \eta \right) & \le \sum_{j=1}^K \eta^{-3}\E[
|\zeta_j-\E[\zeta_j|\mathcal{G}_{j-1}]|^3] \\
& \le C \eta.
\end{align*}
Moreover, since $\E\left[(\zeta_j-\E[\zeta_j|\mathcal{G}_{j-1}])^2|
  \mathcal{G}_{j-1}\right] = O(\eta^4)$, the probability~\eqref{p6} equals
$0$ for $R$ large enough.
Hence $p_6=O(\eta)$.
This shows that
\begin{equation}\label{bm-time}
\PP\left(\max_{1 \le k \le K}|Y_k-\tau_k| > 3\eta|\log \eta|\right)=O(\eta),
\end{equation}
for $R$ large enough.

 Equations~\eqref{mg-time} and~\eqref{bm-time} now imply that
\begin{equation}\label{closetimes}
\PP\left(\max_{1 \le k \le K}|2t_{m_k}-\tau_k| > 6\eta|\log \eta|\right)=O(\eta),
\end{equation}
for $R$ large enough.

Notice that~\eqref{e.bdt} implies that for $k\le K$ ,
\begin{equation}\label{Bdisplacement}
\sup\{|B(t)-B(\tau_{k-1})|: t \in [\tau_{k-1}, \tau_{k}]\} \le 4 \eta,
\end{equation}
and by the definition of $m_k$ and~\eqref{Deltaupperbound} we have for $R$ large enough
\begin{equation*}
\sup\{|\Delta_{m_k}-\vt(t)|: t \in [t_{m_{k-1}}, t_{m_{k}}]\} \le 2 \eta.
\end{equation*}
Summing over $k$ using the definition of $\xi_j$ and $K$, we get from \eqref{DeltaInc} that
\begin{equation*}
\sup\{|\Delta_{m_k}-M_k|: k \le K\} \le c T \eta.
\end{equation*}
By summing, we have $Y_K+t_{m_K} \ge K \eta^2 \ge 10T$. (Consult~\cite{LSW-aop} between~(3.21) and~(3.22) for details.) Hence, the event that $t_{m_K} < 2T$ is contained in the event that $|Y_K- 2t_{m_K}|\ge 4T$. It follows from~\eqref{mg-time} that
\begin{equation}\label{capT}
\PP(t_{m_K}<2T)=O(\eta).
\end{equation}
Set $h=h(\eta)=\eta |\log \eta|$ and consider the event 
\begin{align*}
\mathcal{E} = \{t_{m_K} \ge 2T\} &\cap \left\{\sup_{t \in[0, 2T-h]}
  \sup_{s \in (0, h]} |B(t+s)-B(t)| \le \sqrt{6 h |\log h|}\right\}\\
&\qquad\qquad\cap \left\{\max_{k \le K}|\tau_k-2t_{m_k}| \le
6h\right\}.
\end{align*}
Then in view of the inequalities~\eqref{closetimes},~\eqref{capT}, and Lemma~\ref{moc} (with $\epsilon =1$ and $v=\sqrt{6 |\log h|}$)
we have $\PP(\mathcal{E}^c)=O(\eta |\log \eta|)$. Note that on $\mathcal{E}$ we have that
\begin{align*}
  \sup&\{|\vt(t)-B(2t)|: t\in [0,T]\} \\
  &\le \max_{1\le k \le
    K}\biggl(\sup\{|\vt(t)-\Delta_{m_k}|: t \in [t_{m_{k-1}}, t_{m_k}]\} 
   + |\Delta_{m_k}-B(\tau_{k})| \\
      &\qquad\qquad+ \sup\{|B(\tau_{k})-B(2t)|: t \in [t_{m_{k-1}}, t_{m_k}]\}\biggr),
\end{align*}
and the first two terms are  $O(T \eta)$ uniformly in $k$. For the
last term, we can use~\eqref{Bdisplacement} to see that on $\mathcal{E}$, we have
\begin{align*}
  \sup\{|B(\tau_{k})&-B(2t)|: t \in [t_{m_{k-1}}, t_{m_{k}}]\}\\
  &= \sup\{|B(\tau_{k})-B(s)|: s \in [2t_{m_{k-1}}, 2t_{m_{k}}]\} \\
  &\le  \sup\{|B(\tau_{k})-B(s)|: s \in [\tau_{k-1}-6h,
  \tau_{k}+6h]\}\\
   &\le 4 \eta + \sup\{|B(\tau_{k-1})-B(s)|: s \in [\tau_{k-1}-6h,
  \tau_{k-1}]\} \\
 &\qquad\qquad +\sup\{|B(\tau_{k})-B(s)|: s \in [\tau_{k},
  \tau_{k}+6h]\} \\
   &\le 4 \eta + c (\eta \varphi(1/\eta))^{1/2},
\end{align*}
where $\varphi$ is a subpower function; that is, $\varphi(x)=o(x^{\ee})$ for any $\ee >0$.
It follows that we may couple $\vt$ and $B$ so that 
\[
\PP\left(\sup_{t \in [0,T]}\{|\vt(t)-B(2t)|\} > 
c_1 T\eta^{1/2} \varphi_1(1/\eta) \right) < c_2 \eta | \log \eta|,
\]
where we recall that $\eta(R)=R^{-(1/12-\eps)}$, and $\varphi_1$ is also a subpower
function. This in turn implies that there exist constants $c_1$, $c_2$ such that for every $\eps >0$, all $R$ sufficiently large,
\[
\PP\left(\sup_{t \in [0,T]}\{|\vt(t)-B(2t)|\} > 
c_1R^{-(1/24-\eps)} T \right) < c_2R^{-(1/24-\eps)}.
\]
Together with Proposition~\ref{HM_ROC} this estimate concludes the proof of the theorem.
\end{proof}

\section{Some remarks on the derivation and optimality of the rate}\label{discuss}

We will now briefly review how we obtain the exponent of
$1/24$ in Theorem~\ref{the.theorem2} and how the different scales on
which we work fit together.

There are three main contributions to the exponent. The first comes from the
rate of convergence of the martingale observable as given in
Theorem~\ref{MGthm}, the second comes from
Section~\ref{keysect}, and the third comes from the Skorokhod
embedding step of Section~\ref{Sect-theproof}. There are essentially four
different scales that come into play: the \emph{microscopic} scale which is
basically the lattice size, the \emph{MO} scale corresponding to the
rate of convergence of the martingale observable, the
\emph{mesoscopic} scale on which the discrete
driving function is close to a martingale (with error terms on the MO
scale), and finally the \emph{macroscopic} scale which is of constant order.

Suppose $D \subsetneq \C$  is a simply connected domain with $0
\in D$ and $\inrad(D)=1$, and
 let $D^n$ be the $n^{-1}\Z^2$ grid domain approximation of $D$. Let
 us first consider the rate of convergence of the martingale
 observable, which we denote here by
 $\delta = \delta(n)$; this is the MO scale. The
rate in this step comes essentially from the estimate comparing the
continuous and discrete
 Green's functions in Theorem~\ref{main_green_thm} and is given in
Theorem~\ref{MGthm} to be $\delta(n)=n^{-(1/4-\epsilon)}$. This error
term,
 which we believe to be optimal for the chosen martingale
 observable, then determines a proper
mesoscopic scale on which the driving function corresponding to the
loop-erased random walk path is close to a martingale.  Suppose
that $t_j$ denotes the capacity of $\psi_D(\gamma^n[0,j])$, and let
$\Delta_j=\theta_n(t_j)$. In~\eqref{m}, the integer $m$ is defined
formally, but roughly it is defined so that $t_m \approx \delta^{2/3}$
and $\Delta_m \approx \delta^{1/3}$. Thus $\delta^{2/3}$ is the chosen
mesoscopic scale. Let us remark that the appropriate pairwise
relationship between $t_m$ and $\Delta_m$ is essentially given by
properties of the Loewner equation and the fact that the limiting
driving function is expected to be a Brownian motion which is
H\"older-$h$ for any $h < 1/2$. However, at this
point, the precise relationship between the pair $(t_m, \Delta_m)$ and
$\delta$ need not necessarily be as above. Indeed, for any $\alpha <1$
one could define $m$ so that $\Delta_m \approx \delta^{\alpha/2}$ and
$t_m \approx \delta^{\alpha}$, and the proof of the main estimate in
Section~\ref{keysect} would still work. Choosing some $\alpha > 2/3$
could then in principle improve the error terms in the Skorokhod
embedding step (see below). However, our current argument uses in an
essential way that $\alpha=2/3$.

The penultimate goal is to show that the Loewner driving function of a
macroscopic piece of the loop-erased random walk path is close to a Brownian
motion. The domain Markov property of loop-erased random walk allows
us to iterate the estimates for mesoscopic pieces of the curve to
``build'' a macroscopic piece of it. This is done in the final step of
the proof, where the
Skorokhod embedding scheme is employed.
The convergence rate in this step is essentially determined by the
maximal step size of the discretized driving function, that is, the
magnitude of $|\Delta_m|$ which we here write as $\delta^{1/3}$. The
resulting error term after Skorokhod embedding is then roughly
\begin{equation}\label{optimality}
\delta^{(1/3)(1/2)}=O(n^{-(1/4)(1/3)(1/2)+\epsilon)}) = O(n^{-(1/24-\eps)}).
\end{equation}
We believe this is optimal (up to
subpower correction), given the magnitude of $|\Delta_m|$. This is
because it is known (see~\cite{Borovkov} for details) in the ``nicer'' case of
(one-dimensional) simple random walk that the Skorokhod embedding
scheme gives a rate of the square root of the step size of the
normalized random walk, in agreement with our estimates.

As mentioned above, we believe that the exponents $1/4$ and $1/2$ in~\eqref{optimality} are optimal. The conjectured optimality of the
$1/4$ exponent only pertains to our specific choice of martingale
observable. Our general method may allow for some improvement of the
exponent $1/3$, as long as it remains smaller than $1/2$. We also believe that much of our work can be used for the
derivation of the rate of convergence for other processes known to
converge to SLE, in the sense that it should be usable as it is or
with minor modifications once one has a rate of convergence for an
appropriate martingale observable; see~\cite{FJVROC}. This should also be true for the
derivation of the rate of convergence with respect to Hausdorff
distance, at least as long as the limiting curve is simple.

\section{Hausdorff convergence} \label{Sect-hausdorff}

Recall that the Hausdorff distance between two compact sets $A$, $B \subset \C$ is defined by
\[
d_H(A,B)=\inf\left\{\ee >0: A \subset \bigcup_{z \in B}\ball(z, \ee), \, B \subset \bigcup_{z \in A}\ball(z, \ee)\right\}.
\]
In this section we prove a rate of convergence result for the pathwise convergence with respect to Hausdorff distance. 
We use the notation from Section~\ref{Intro} and in addition we let $\tilde{\gamma}$ denote the radial SLE$_2$ path in $\Disk$ started from $1$. In this section it is convenient to parameterize $\tilde{\gamma}$ by \emph{half-plane capacity}. For the rest of this section we let 
\[
\vp(z):=\frac{i-z}{i+z}.
\] 
Note that $\vp$ maps the upper half-plane conformally onto the unit disk with $\vp(i)=0$, $\vp(0)=1$, and $\vp(\infty)=-1$.
We parameterize $\tilde{\gamma}(t)$, $0 \le t \le t_0$, so that the image in $\Half$ satisfies
\[
\hcap[\vp^{-1}(\tilde{\gamma}[0,t])]=2t.
\] 
We assume that $t_0$ is sufficiently small so that $\ecap(\tilde{\gamma}[0,t_0])$ is bounded by a constant a.s., specifically we may take, for example, $t_0=1/16$. (By comparing with the chordal Loewner chain driven by a constant driving function one can see that the conformal radius of $\mathbb{H} \setminus \vp^{-1}(\tilde{\gamma}[0,t])$ seen from $i$ is bounded below if $\hcap[\vp^{-1}(\tilde{\gamma}[0,t])]$ is chosen sufficiently small.) We allow all constants in this section to depend on $t_0$.

Our approach uses a uniform derivative estimate for radial SLE$_2$ that we derive from an estimate on the growth of the derivative of the chordal SLE mapping from~\cite{rohde_schramm}. This is the reason why it is convenient to use the half-plane capacity parameterization. 

\begin{theorem}\label{hd.thm} 
Let $0<t \le t_0$ where $t_0$ is sufficiently small. There exists $c< \infty$ with the property that for $n$ sufficiently large there is a coupling of $\tilde{\gamma}^n$ with $\tilde{\gamma}$ such that
\[
\PP\left(d_H\left(\tilde{\gamma}^n[0,t] \cup \partial \Disk, \tilde{\gamma}[0,t]  \cup \partial \Disk\right) > c(\log n)^{-p} \right) < c(\log n)^{-p}
\]
whenever $p<(15-8\sqrt{3})/66$.
 \end{theorem}
\begin{remark*}
 We believe that it is possible, with substantially more work, to use a strategy similar to the one in this section to obtain a \emph{power-law} convergence rate for the image of the LERW path in $\Disk$ as a (parameterized) curve. This is the object of another paper in progress (\cite{FJVROC}). 
 \end{remark*}
We will first consider a deterministic setting with two solutions to the radial Loewner equation driven by functions which are at uniform distance at most $\ee>0$. Using the reverse-time (radial) Loewner equation it is not hard to quantify, in terms of $\ee$, the uniform distance between the solutions restricted to a set at distance at least $\delta=\delta(\ee)$ from the boundary of the unit disk. If some additional information about the growth of the derivative of one of the Loewner maps is known, this can be used to estimate the Hausdorff distance between the boundaries. The details are given in Lemma~\ref{sept26.lemma}. We will later show that the assumptions of this lemma are satisfied on an event of large probability in the coupling of Theorem~\ref{the.theorem2}.
\begin{lemma}\label{sept26.lemma}
Let $0< T < \infty$ be fixed. For $j=1$, $2$, let 
\[h_j(t,z): \D \to \D \setminus \gamma_j[0,t]\] 
be the solution to the radial Loewner equation generated by the simple curve $\gamma_j$ such that $\partial h_j(\mathbb{D}) = \partial \mathbb{D} \cup \gamma_j[0,T]$ with $W_j$ as driving function and write $h_j(z)=h_j(T,z)$.  Suppose that
\[
\sup_{0\le t \le T}|W_1(t)-W_2(t)| < \ee,
\]
where $\ee>0$ is sufficiently small.
Suppose further that there exists $0< \beta <1$ such that for all $\zeta \in \partial \D$,
\begin{equation}\label{sept26.3}
|h_1'((1-\delta)\zeta)| \le \delta^{-\beta} \textrm{ whenever } \delta < |\log \ee|^{-1/2}.
\end{equation}
Then there is a constant $c < \iy$ depending only on $T$ and  $\beta$ such that
\[
d_H(\gamma_1[0,T] \cup \partial \D, \, \gamma_2[0,T] \cup \partial \D) \le 
c\,|\log \ee|^{-(1-\beta)/2}.
\]
\end{lemma}
  
\begin{proof}
Let us first note that our assumptions imply that for $j=1,2$, $h_j$ extends continuously to the closure of $\mathbb{D}$. Indeed, $\gamma_j[0,T]$ is by assumption a (simple) curve and is therefore locally connected. Since $\partial \mathbb{D}$ is clearly locally connected, and finite unions of locally connected sets are again locally connected, $\partial h_j(\mathbb{D})$ is locally connected. Theorem~2.1 of \cite{pommerenke} then implies the claimed continuity up to the boundary.

For $j=1,2$ and $0 \le t \le T$, let $\tilde{h}_j(t,z;T)$ be the solution to the reverse-time radial Loewner equation driven by $t \mapsto W_j(T-t)$:
\[
\partial_t\tilde{h}_j(t,z;T) = -\tilde{h}_j(t,z;T) \frac{W_j(T-t)+\tilde{h}_j(t,z;T)}{W_j(T-t)-\tilde{h}_j(t,z;T)}, \quad \tilde{h}_j(0,z;T)=z.
\] 
It is well-known that $\tilde{h}_j(T,z;T) = h_j(T,z)=:h_j(z)$ for $z \in \mathbb{D}$.
Let \[H(t)=\tilde{h}_1(t,z;T)-\tilde{h}_2(t,z;T), \quad 0 \le t \le T,\] so that $|H(T)|=|h_1(z)-h_2(z)|$. Then by differentiating $H(t)$ with respect to $t$, using the reverse-time Loewner equation, we obtain a linear ODE which can be solved using the method of integrating factor. Using the fact that points flow away from $\partial \Disk$ under the reverse flow, it is then easy to see that there is a $c_0<\iy$ depending only on $T$ such that
\begin{equation*}
|h_1(z)-h_2(z)|=|H(T)| \le \ee \left(e^{c_0/\delta^2}-1\right)
\end{equation*}
whenever $|z| \le 1-\delta$. A similar calculation (using the forward-time Loewner equation) may be found in Section~4.7 of~\cite{SLEbook}.

Hence, by taking $\delta = \delta_0:=|4c_0/\log \ee|^{1/2}$, we see that 
\begin{equation}\label{sept26.1}
\sup_{|z| \le 1-\delta_0/2}|h_1(z)-h_2(z)| \le c \ee^{1/2}.
\end{equation}
Using this, Cauchy's integral formula implies 
\begin{equation}\label{sept26.2}
\sup_{|z| \le 1-\delta_0}|h'_1(z)-h'_2(z)| \le \sup_{|z| \le 1-\delta_0} c \int_{\partial \ball(z, \delta_0/2)}\frac{\ee^{1/2} d\zeta}{|z-\zeta|^2}\le c (\ee |\log \ee|)^{1/2}.
\end{equation}
For ease of notation we shall write \[\hat\gamma_j=\gamma_j[0,T]\cup \partial \D, \quad j=1,2,\] 
in the sequel.
Fix $\zeta \in \partial \D$. Then, using \eqref{sept26.3} and the Koebe distortion theorem (Lemma~\ref{Koebedistthm}) we can integrate $h_1'$ up to the boundary to get
\begin{equation*}
|h_1(\zeta)-h_1((1-\delta)\zeta)| \le c\delta_0^{1-\beta}, \quad 0 \le \delta \le \delta_0. 
\end{equation*}
Hence, for some constant $c< \infty$,
\[
h_1( \{1-\delta_0 \le |z| \le 1\}) \subset \bigcup_{z \in \hat\gamma_1}\ball(z, c\delta_0^{1-\beta}),
\]
where $h_1(z)$ when $|z|=1$ is defined by taking a limit using the continuity on $\overline{\mathbb{D}}$.
It follows in view of~\eqref{sept26.1} that, with perhaps a different $c$, we can replace $h_1$ by $h_2$ in the last expression, again using continuity.
Clearly $\hat \gamma_2 \subset h_2( \{ 1-\delta_0 \le |z| \le 1\})$, so it remains to show that
\begin{equation}\label{gamma1gamma2}
\hat\gamma_1 \subset \bigcup_{z \in \hat \gamma_2}\ball(z, c \delta_0^{1-\beta}).
\end{equation}
Write $w=\lim_{\delta \to 0}h_1((1-\delta)\zeta) \in \hat \gamma_1$. Let $\hat{w} \in \hat \gamma_2$ be a point closest to $h_2((1-\delta_0)\zeta)$. Then by Koebe's estimate,~\eqref{sept26.2}, and~\eqref{sept26.3} we have
\begin{align*}
|\hat{w} - h_2((1-\delta_0)\zeta)| & \le c\, \delta_0 | h_2'((1-\delta_0)\zeta)| \\
& \le c \, \delta_0\left(|h_1'((1-\delta_0)\zeta)| + c'(\ee |\log \ee|)^{1/2} \right)\\
& \le c \, \delta_0^{1-\beta}.
\end{align*}
Hence, using~\eqref{sept26.1} and~\eqref{sept26.3},
\begin{align*}
|\hat{w}-w|  & \le  |\hat{w}-h_2((1-\delta_0)\zeta)| + |h_2((1-\delta_0)\zeta)-h_1((1-\delta_0)\zeta)|\\
 &\qquad + |h_1((1-\delta_0)\zeta)-w|\\
& \le  c\left(\delta_0^{1-\beta} + \ee^{1/2} + \delta_0^{1-\beta}  \right),
\end{align*}
and this implies that~\eqref{gamma1gamma2} holds,
which concludes the proof.
\end{proof}
The idea is now to use the coupling the existence of which follows from Theorem~\ref{the.theorem2} and an estimate on the growth of the derivative of the radial SLE$_2$ mapping $f$ to find an event with large probability on which Lemma~\ref{sept26.lemma} can be applied. We will derive the needed estimate from the corresponding estimate for standard \emph{chordal} SLE$_2$ in~\cite{rohde_schramm}.    
For this, we need to change to ``chordal coordinates'' and express the derivative of $f$ in terms of a suitable conformal map $F$ defined in $\Half$. We shall then see that $|F'|$ can be estimated in terms of the derivative of the conformal map associated with standard chordal SLE$_2$ (at a fixed half-plane capacity time). 

Let us start by mapping the curve $\tilde{\gamma}$ to $\Half$ via the M\"obius transformation $\vp^{-1}$ defined in the beginning of the section, and note that there exists a unique conformal map $F_t: \Half \to \Half \setminus \varphi^{-1}(\tilde{\gamma}[0,t])$ satisfying the hydrodynamical normalization, that is, 
\[
F_t(z)=z-2t/z + O(1/z^2), \quad z \to \infty,
\] 
where $2t$ is the half-plane capacity of $\varphi^{-1}(\tilde{\gamma}[0,t])$ and we assume that $t \le t_0$. It is known that the mapping $F_t$ has the distribution of a chordal SLE$(2;-4)$ mapping with force point $i$ (evaluated at the fixed time $t$), see \cite{SW}. (As we only need to know that chordal SLE$(2;-4)$ is mutually absolutely continuous with respect to chordal SLE$_2$ for sufficiently small times, we will not further discuss SLE$(\kappa; \rho)$ and refer the reader to~\cite{SW} for definitions and additional details.) This also means that $F_t$ can be viewed as a chordal SLE$_2$ mapping weighted by a certain local martingale in a sense which we now briefly explain. (See \cite{SW} for the proof.) Let $(\hat{G}_t)$ be the standard chordal SLE$_2$ Loewner chain, that is, $\hat{G}_t$ satisfies the chordal (forward time) Loewner equation driven by $\sqrt{2} B_t$, where $B$ is standard Brownian motion. Set $\hat{F}_t=\hat{G}_t^{-1}$. Let $x_t+iy_t=\hat{G}_t(i)-\sqrt{2}B_t$. Then by applying It\^o's formula one can see that $M_t=y_t/(x_t^2+y_t^2)$ is a local martingale with respect to the chordal SLE$_2$ path. Note that $M_0=1$. By considering a stopped version of $M$ (or $M$ evaluated at a sufficiently small deterministic $t_0$, which is the case we consider) we can define a new probability measure $\mathbb{E}[ \one\{\cdot\} M_{t_0}]$ by \emph{weighting} by $M$. Let $U_t$ be the chordal driving function for $F_t$, $0 \le t \le t_0$, that is, $U_t$ is the unique continuous real-valued function such that for $0 \le t \le t_0$,
\[
\partial_t F_t(z) = -\partial_z F_t(z) \frac{2}{z-U_t}, \quad F_0(z)=z, \quad z \in \Half.
\]
By Girsanov's theorem, under the measure obtained by weighting by $M_{t_0}$, a standard linear Brownian motion with speed 2 has the same distribution as the process $U_t$, $0 \le t \le t_0$, under the ``unweighted'' measure.
Since $y_0=1$ the Loewner equation implies that $y_t \ge \sqrt{1-4t}$ if $t <1/4$. Consequently, if $0 \le t \le t_0 < 1/4$ we see that $M_t \le 1/\sqrt{1-4t_0}$ and we can write 
\begin{align}\label{jul9.1}
\Prob(|F_t'(x+iy)| \ge   y^{-\beta})
&=\E[ \one\{|\hat{F}_t'(x+iy)| \ge  y^{-\beta}\} M_{t_0}] \nonumber\\
&\le  \frac{\Prob(|\hat{F}_t'(x+iy)| \ge y^{-\beta})}{\sqrt{1-4t_0}},
\end{align}
where $\beta<1$. The last probability can be estimated using Corollary~3.5 of~\cite{rohde_schramm}.

We now show that the radial SLE$_2$ mapping 
\[
f_t: \Disk \to \Disk \setminus \tilde{\gamma}[0,t], \quad f_t(0)=0, \;\; f_t'(0)>0,
\]  
can be expressed in terms of $F_t$. Indeed, we can write
\begin{equation}\label{feb25.1}
f_t=\varphi \circ F_t \circ \Delta,
\end{equation}
where, for $G_t=F_t^{-1}$,
\[
\Delta(z)=\Delta_t(z)=\frac{z \overline{G_t(i)} -\lambda G_t(i) }{z - \lambda} 
\]
is a random Moebius transformation accounting for the fact that $f_t$ and $F_t$ are normalized at an interior point and at a boundary point, respectively. Note that $\Delta: \Disk \to \Half$ and $\Delta(0)=G_t(i)$, $\Delta(\lambda) = \infty$, where $\lambda=f_t^{-1}(-1)$. Indeed, if $|\lambda|=1$ then the right hand side of~\eqref{feb25.1} maps $\D$ onto $\D \setminus \tilde{\gamma}[0,t]$ and fixes the origin where it has derivative $[iF'(G(i))\imag G(i)]/
\lambda$. Hence we can choose a unique $\lambda$ with $|\lambda| =1$ such that this derivative is real and positive. But this defines $f$ by uniqueness of Riemann mappings. We can then see that $f(\lambda)=-1$.  

The chain rule gives
\begin{equation}\label{jul5.3}
|f_t'(z)| = |\vp'(F_t(\Delta(z))| | F_t'(\Delta(z))| |\Delta'(z)|.
\end{equation}
Note that 
\begin{equation}\label{jul5.4}
|\vp'(F_t(\Delta(z)))||\Delta'(z)|=\frac{2 }{|F_t(\Delta(z))+i|^2} \frac{|2 \lambda \imag G_t(i)|}{|z-\lambda|^2} \le c_0, \quad z\in \Disk, 
\end{equation}
where the inequality comes from performing a series expansion using the hydrodynamical normalization. 

We can now derive the needed estimate for the growth of the derivative of the radial SLE$_2$ mapping $f_t$. For the purpose of stating the lemma, define $\beta_0 = (14 + 4\sqrt{6})/25$ and set
\begin{equation}\label{rho}
\rho(\beta)=\frac{25 \beta^2-28\beta+4}{8 \beta}.
\end{equation}

\begin{lemma}\label{v-lemma}
There exists a subpower function $\phi$ such that for $\beta_0 < \beta <1$ and all sufficiently small $\delta_0 > 0$
\begin{equation*}
\PP \left(\left\{\sup_{\zeta \in  \partial \D}|f_{t_0}'((1-\delta)\zeta)| \le \delta^{-\beta} \text{ for all } \delta < \delta_0 \right\}^c\right) \le \phi(1/\delta_0)\delta_0^{\rho(\beta)}.
\end{equation*}
\end{lemma}

\begin{proof}
Recall the definition of $U_t$ as the chordal driving function for $F_t$. For notational simplicity we will write $f=f_{t_0}$ and $F=F_{t_0}$. Let $c_1$ be fixed for the moment and let $\mathcal{V}=\mathcal{V}_{\delta_0}$ be the event that
\[
\sup_{0 \le t \le t_0}|U_t| \le c_1 \sqrt{ \log \delta_0^{-1}}.
\]
The idea is that on this event the curve in $\mathbb{H}$ is contained in a box with an aspect ratio which is not too large. This means that we can control the change of variables connecting $f$ with $F$ up to a logarithmic correction; the precise estimate is given in \eqref{july8.3} below. Now for the details. 
Recall that $U_t$ has the same distribution as a Brownian motion weighted by the local martingale $M$. Consequently, by our assumption that $t \le t_0$, if $c_1$ is chosen sufficiently large, we can compare with a Brownian motion using the reflection principle to see that 
\begin{equation}\label{jul8.1}
\PP(\mathcal{V}^c)=O(\exp\{-\log \delta_0^{-1}\})=O(\delta_0).
\end{equation}

We claim that there are constants $0< c, c'<\infty$ such that if
\begin{equation}\label{A}
A=A(c)=\left[-c \sqrt{\log \delta_0^{-1}}, c\sqrt{\log \delta_0^{-1}}\, \right] \times (0, c \delta_0 \log \delta_0^{-1}\,],
\end{equation} then 
\begin{align}\label{july8.3}
\PP &\left( \left\{\sup_{\zeta \in  \partial \D}|f'((1-\delta)\zeta)|
\le \delta^{-\beta}, \, \delta < \delta_0 \right\} ^c \cap \mathcal{V} \right) \nonumber\\
&\qquad\qquad\le \PP \left(\exists w \in A: |F'(w)| \ge c'[\imag w \log \delta_0^{-1}]^{-\beta}\right).
\end{align}
To prove~\eqref{july8.3}, write $\hat{\gamma}(t)=\vp^{-1}(\tilde{\gamma}(t))$ for the image of the radial SLE$_2$ path in $\Half$ parameterized by half-plane capacity. Note that the chordal Loewner equation implies that on $\mathcal{V}$
\[
\hat{\gamma}[0,t_0] \subset \left[-c_1 \sqrt{\log \delta_0^{-1}}, c_1 \sqrt{\log \delta_0^{-1}}\,\right] \times [0,2\sqrt{t_0} \,].
\]
Consequently, if $\delta_0$ is sufficiently small, by estimating the harmonic measure of the rectangle $[-c_1 \sqrt{\log \delta_0^{-1}}, c_1 \sqrt{\log \delta_0^{-1}}] \times [0,2\sqrt{t_0}] \subset \mathbb{H}$ we see that there exists $c_2 < \infty$ such that the pre-image of $\hat{\gamma}[0,t_0]$ in $\R$ satisfies
\begin{equation}\label{july8.2}
G(\hat{\gamma}[0,t_0]) \subset \left[-c_2 \sqrt{\log \delta_0^{-1}}, c_2 \sqrt{\log \delta_0^{-1}}\,\right]
\end{equation}
on $\mathcal{V}$. We can assume that $c_2 \ge c_1$.
Define
\[
L:=2c_2 \sqrt{\log \delta_0^{-1}}, \quad I:=[-L,L],
\]
where $c_2$ is as in~\eqref{july8.2}. (Note that on $\mathcal{V}$ we have that $f^{-1}(\tilde{\gamma}[0,t_0]) \subset \Delta^{-1}(I/2) \subset \partial \D$.)
The normalization of $F$ at infinity, Schwarz's reflection principle, and Koebe's distortion theorem (Lemma~\ref{Koebedistthm}) imply that there exists a constant $0 < c<\infty$ such that
\[
c^{-1} \le |F'(x+iy)| \le c , \quad x \in \R \setminus I,
\] 
holds on the event $\mathcal{V}$. Indeed, if $S(z)=(z+1/z)/4$, then since the function $z \mapsto 4F(LS(z))/L$ belongs to the class $\Sigma$ of normalized conformal mappings of the exterior unit disk it has a uniformly bounded derivative for $|z| >2$. (Here we have extended $F$ to $\C \setminus [-L/2, L/2]$ by Schwarz reflection without changing its symbol.) 
Moreover, since $t \le t_0$, the Loewner equation implies that $\imag G(i) \asymp 1$ and that there is a constant $c_3 < \infty$ such that
\[
|\real G(i) | \le c_3.
\]
Further, there is a constant $c_4 >0$ such that
\[
|\Delta^{-1}(x)-\lambda|=\frac{2 \imag G(i)}{|x-\overline{G(i)}|} \ge \frac{c_4}{L}, \quad x \in I.
\]  
Note also that $\imag \Delta(z) = \imag G(i)(2\delta - \delta^2)/|z-\lambda|^2$. Consequently,  we can choose a constant $c <\infty$ for $A=A(c)$ (the rectangle defined in \eqref{A}) such that
\[
\Delta\left(\{(1-\delta)\zeta: 0 \le \delta \le \delta_0,\, \zeta \in \Delta^{-1}(I)\}\right) \subset A 
\]
and so~\eqref{july8.3} follows from~\eqref{jul5.3} and~\eqref{jul5.4}. We assume $A$ is defined using this particular $c$ from now on.

By~\eqref{jul8.1} it remains to estimate the right-hand side of~\eqref{july8.3}. Thus we wish to estimate the probability that there is a point $w$ in the rectangle $A$ such that the derivative $|F'(w)|$ exceeds $(\imag w)^{-\beta}$ (with a logarithmic correction). To this end, we consider a Whitney decomposition of the rectangle $A$, that is, we take a partition using dyadic rectangles of uniformly bounded hyperbolic diameter:
\[
S_{j,k}=\{x+iy: 2^{-j-1} \le y \le 2^{-j}, \,  (k-1)2^{-j} \le x \le k  2^{-j} \}, 
\]
for integer $j \ge \lfloor -\log (c \delta_0 L^2) \rfloor$ and $- L 2^j \le k \le L 2^j$, where $c$ is the constant used in the definition of $A$. Then $A \subset \cup S_{j,k}$ where the union is over the above $j$, $k$.  Let $z_{j,k}$ be the center point of $S_{j,k}$. By the Koebe distortion theorem (Lemma~\ref{Koebedistthm}) $|F'(z)|/|F'(z_{j,k})| \asymp 1$ for $z \in S_{j,k}$. Hence, to estimate the probability of the existence of a $w \in A$ where the derivative of $F$ is large it is enough to restrict attention to the centers $z_{j,k}$ of the Whitney cubes. We will estimate the derivative in each of the $z_{j,k}$ using Chebyshev's inequality and an estimate on the moments of $|F'|$ from~\cite{rohde_schramm}. 

More precisely, the last paragraph, Chebyshev's inequality, and the moment estimate Corollary~3.5 of~\cite{rohde_schramm} show that
\begin{align*}
\PP&\left(\exists w \in A: |F'(w)| > c'(L^2\imag w)^{-\beta}\right) \\
&\qquad\qquad\le \sum_{j=\lfloor  -\log(c \delta_0 L^2) \rfloor}^{\iy} \sum_{k=-L2^j}^{L 2^j} \PP\left(|F'(z_{j,k})| \ge c(2^{j}/L^2)^{\beta}\right) \\
&\qquad\qquad\le \phi(1/\delta_0)\sum_{j=\lfloor  -\log (c\delta_0 L^2) \rfloor}^{\iy} (2^{-j})^{\rho_{b}(\beta)}.
\end{align*}
Here $\rho_{b}(\beta):=-1-2b+\beta b(5-2b)$ with $b \in [0,3]$ (see~\cite{rohde_schramm} for more details) and $\phi$ is a subpower function. When $\rho_b(\beta) >0$ the last term is finite and bounded above by $\phi_1(1/\delta_0)\delta_0^{\rho_b(\beta)}$ for a possibly different subpower function $\phi_1$. We maximize over $b$ to find $\rho(\beta)=(25 \beta^2-28\beta+4)/(8 \beta)$, which is in $(0,1)$ when $ 1>\beta > \beta_0=(14 + 4\sqrt{6})/25$.
\end{proof}

\begin{proof}[Proof of Theorem~\ref{hd.thm}] Let $0<T<\infty$ be fixed for the moment.
Set $\ee_n=n^{-(1/24-\ee)}$ for some fixed $0<\ee<1/24$. By Theorem~\ref{the.theorem2}, if $n$ is sufficiently large, there is a coupling of $\gamma^n$ with Brownian motion $B$ started uniformly on $\partial \D$ and an event $E_n$ with $\PP(E_n^c)<\ee_n$
on which 
for sufficiently large $n$ the estimate
\[
\sup_{0 \le t \le T}|W_n(t)-e^{iB(2t)}| \le \ee_n,
\]
holds, where $W_n$ is the radial Loewner driving function for $\tilde{\gamma}^n$. We extend the coupling to include $\tilde{\gamma}$, the radial SLE$_2$ path, which is obtained deterministically from $B$. 
Let $f: \D \to \D \setminus \tilde{\gamma}[0,t]$ be the radial SLE$_2$ mapping evaluated at a fixed half-plane capacity time $t \le t_0$ as above. We can assume that $T > \ecap(\tilde{\gamma}[0,t_0])$. Set 
\[ \delta_n= |\log \ee_n|^{-1/2},\]
and let $\mathcal{F}_n(\beta)$ be the event that
\[
\sup_{\zeta \in \partial \D}|f'\left((1-\delta)\zeta \right)| \le \delta^{-\beta}, \quad \delta \le \delta_n.
\]
By Lemma~\ref{v-lemma} there exists a subpower function $\phi$ such that
\[
\PP(\mathcal{F}_n(\beta)^c) \le \phi(1/\delta_n) \delta_n^{\rho(\beta)},
\]
where $\beta_0<\beta < 1$ and $\rho$ is given in~\eqref{rho}.

On the event $E_n \cap F_n(\beta)$, which has probability at least $1-\phi(1/\delta_n)\delta_n^{\rho(\beta)}$, we can apply Lemma~\ref{sept26.lemma} to see that 
\[
d_H(\tilde{\gamma}^n[0,t] \cup \partial \D, \, \tilde{\gamma}[0,t] \cup \partial \D) \le c\delta_n^{1-\beta}.
\] 
We solve $\rho(\beta)=1-\beta$ for $\beta_0 < \beta < 1$ to find $\rho(\beta)=(15 - 8 \sqrt{3})/33 \approx 0.035$ and this completes the proof.
\end{proof}

\appendix

\section{An estimate for Green's functions}\label{greenappendix}

In this appendix, we derive a uniform bound for the difference between the discrete and continuous Green's functions in a simply connected grid domain $D\in \grid$ with $\inrad(D) =R$.  This bound, given in Theorem~\ref{main_green_thm}, is an essential ingredient in the proof of Proposition~\ref{KLtheorem}, which is given in Appendix~\ref{prop42appendix}. The bound we provide is, to our knowledge, an improvement over existing bounds of the same kind. In particular, Theorem~\ref{main_green_thm} is an improvement of Theorem~1.2 of~\cite{KozL1} where an error term of $R^{-1/3}\log R$ was obtained by using a rougher decomposition. Throughout this section, $B$ will denote a planar Brownian motion and $S$ will denote a simple random walk on $\Z^2$.  Recall that the continuous Green's function $g_D(x,y)$ and the discrete Green's function $G_D(x,y)$ are defined in Section~\ref{greenZ}.

\begin{theorem}\label{main_green_thm}
Let $0<\ee<1/2$ and let $0<\rho<1$ both be fixed. There exists a
constant $c$ such that the following holds. Suppose that $D\in\grid$
is a grid domain with $\inrad(D)=R>0$, and let $V =V(D)= D \cap
\Z^2$. If $x$, $y \in V$ with $x \neq y$ and $|\psi_D(x)| \le \rho$, then
\begin{equation}\label{EJPmar29.eq1}
\left| G_D(x,y) - \frac{2}{\pi} \, g_D(x,y) - k_{y-x} \right| \le c
R^{-(1/2-\ee)}
\end{equation}
where
$$k_z = k_0 + \frac{2}{\pi}\, \log|z| - a(z),$$
$a(z)$ is the potential kernel as in~\eqref{1.2},  and $k_0 =
(2\varsigma + 3\ln 2)/\pi$ where $\varsigma$ is Euler's constant.
\end{theorem}

In what follows, the symbol $\PP^x$ will refer to the measure of random walk or Brownian motion (or both) started at $x$. There will be no ambiguity as to which process is referred to by $\PP^x$. The symbol $\E^x$ will represent the expected value associated with $\PP^x$. If $x=0$, we just write $\PP$ and $\E$. The proof of Theorem~\ref{main_green_thm} relies on a fine estimate of  
$|\E^x[\log|B_T|] - \E^x[\log|S_{\tau}|]|$
 when $B$ and $S$ are linked by a strong coupling, where $T$ and $\tau$ are the exiting times from $D$ by $B$ and $S$, respectively. The majority of this section is devoted to the derivation of this estimate. Before stating it, we introduce a symbol which we will use throughout the section to alleviate the otherwise cumbersome notation. If for two functions $f$ and $g$ defined on some domain $A$ there exists a constant $c>0$ such that  $f(x)\leq cg(x)$ for all $x\in A$, we will write $f(x) \lesssim g(x)$. It will always be clear from context what the variable and the domain are.

\begin{lemma}\label{exitestimate.nowlemma}
For any $\ee>0$ there exists a constant $c$ such that if $D\in\grid$
is a grid domain with $\inrad(D)=R>0$, $B$ is a planar standard
Brownian motion, and $S$ is a two-dimensional simple random walk,
then for any $x\in\Z^2$ with $|x|\leq R^2$,
$$\big|\,\E^x[\log|B_T|] - \E^x[\log|S_{\tau}|]\,\big|\leq c
R^{-(1/2-\ee)},$$
where $T=\inf\{t\geq 0: B_t\not\in D\}$ and $\tau=\inf\{k\geq
0:S_k\not\in D\}$.
\end{lemma}

\begin{remark*} We believe that the bound $R^{-(1/2-\ee)}$ in Lemma~\ref{exitestimate.nowlemma} can be replaced by $(|x|\vee R)^{-(1/2-\ee)}$ and give a heuristic argument for it in the proof below, but prove the weaker version, as it is all we need.
\end{remark*}

At the heart of the proof lies a coupling argument, the strong approximation of Koml\'os, Major, and Tusn\'ady, which we state here and which we will refer to below as the KMT coupling. For a proof of the one-dimensional case, see~\cite{kmt2} and for the two-dimensional case, see~\cite{auer}. Note that this version of the KMT coupling is for fixed times, while the version stated in Lemma \ref{KMTthm1} is for specific stopping times.

\begin{theorem}\label{kmt}
There exists a coupling of planar Brownian motion $B$ and two-dimensional simple random walk $S$ with $B_0=S_0$, and a constant $c>0$ such that for every $\lambda >0$, every $n\in\R_+$, 
$$\Prob \bigg(\sup_{0\leq t\leq n}|S_{2t}-B_t| > c(\lambda + 1)\log n\bigg)\leq cn^{-\lambda}.$$ 
\end{theorem}

In Theorem~\ref{kmt}, $S$ represents random walk interpolated linearly between integer times. For the rest of this section, we use the same notation and in addition run the random walk at twice the Brownian speed. This way, on the probability space of Theorem~\ref{kmt}, it is $B_t$ and $S_t$ that are close, rather than $B_t$ and $S_{2t}$.

A few other technical ingredients will be needed in order to cook up the proof of Lemma~\ref{exitestimate.nowlemma}. Among them are Beurling estimates (see Section~\ref{Sect-notation}) and the following large
deviations estimates giving an upper bound for the probability that in
time $n$ random walk or Brownian motion travel much beyond distance
$\sqrt{n}$ or remain in a disk of radius much smaller than $\sqrt{n}$. For the proofs, see~\cite{benesnotes}.

\begin{lemma}\label{largedev} If $B$ is a planar Brownian motion and $S$ is a
  planar simple random walk, there exists a constant $c<\infty$ such that for every
  $n\geq 0$, every $r\geq 1$,
\begin{equation*}
\Prob\bigg(\sup_{0\leq t \leq n} |B_t|\geq r\sqrt{n}\bigg) \leq
c\exp\left\{-r^2/2\right\}
\end{equation*}
and
\begin{equation*}
\Prob\bigg(\max_{0\leq k \leq 2n} |S_k|\geq r\sqrt{n}\bigg) \leq
c\exp\left\{-r^2/4\right\}.
\end{equation*}
\end{lemma}

\begin{lemma}\label{constrained} If $B$ is a planar Brownian motion and $S$ is a
  planar simple random walk, there exists a constant $c>0$ such that for every
  $n\geq 0$, every $r\geq 1$,
\begin{equation*}
\Prob\bigg(\sup_{0\leq t \leq n} |B_t|\leq r^{-1}\sqrt{n}\bigg) \leq \exp\left\{-cr^2\right\}
\end{equation*}
and
$$\Prob\bigg(\max_{0\leq k \leq 2n} |S_k|\leq r^{-1}\sqrt{n}\bigg) \leq \exp\left\{-cr^2\right\}.$$
\end{lemma}

\begin{proof}[Proof of Lemma~\ref{exitestimate.nowlemma}] 
We assume throughout the proof that $B$ and $S$ are coupled as in Theorem~\ref{kmt}. Recall that in this section, $S$ represents random walk run at twice the usual speed. 
The main difficulty in deriving the optimal estimate lies in the fact that there are a number of different kinds of configurations that can make $\big|\log|B_T| - \log|S_{\tau}|\,\big|$ very large, for instance the rare events of Theorem~\ref{kmt} or the event that $T$ is unusually large. To obtain a precise estimate, we fix $\eps >0$, consider a grid domain $D$ with $\inrad(D)=R>0$, and define
\begin{align*}
\sigma & = \min\{T,\,\tau\},\\
\mathcal{A}_k & =  \big\{|B_T| \in [R^{1+k\eps}, R^{1+(k+1)\eps})\big\}, \;\;\; k\geq 0,\\
\mathcal{B}_{\ell} & =  \big\{\sigma\in [R^{\ell\eps}-1, R^{(\ell+1)\eps}-1)\big\}, \;\;\;\ell \geq 0,\\
\mathcal{C}_m & =  \big\{|B_T-S_{\tau}|\in[R^{m\eps}-1,R^{(m+1)\eps}-1)\big\}, \;\;\;m\geq 0,\\
\mathcal{H}_{r,\ell} & =  \bigg\{\sup_{0\leq t \leq \sigma}|B_t-S_t|\in [cr\log R^{(\ell+1)\eps},c(r+1)\log R^{(\ell+1)\eps})\bigg\}, \;\;\;
r, \ell \geq 0,
\end{align*}
where $c$ is  the constant from Theorem~\ref{kmt}.

Suppose that $d:=\dist(x,\bd D)$, and assume $s=s(\ee, R, d)$ is defined so that $d=R^{s\eps}$. In other words, $s=(\log d)/(\ee \log R)$.
Then, clearly, 
\begin{align}\label{maineq}
\big|\,&\E^x[\log|B_T|]  - \E^x[\log|S_{\tau}|]\,\big| \nonumber \\ 
 & \leq \sum_{k, \ell, m, r\geq 0} \E^x\bigg[\big|\log|B_T| - \log|S_{\tau}|\big|\one\{\mathcal{A}_k\cap \mathcal{B}_{\ell} \cap \mathcal{C}_m \cap \mathcal{H}_{r,\ell}\}\bigg].
\end{align}
 The dominant term in the sum above corresponds essentially to the terms where $T \approx d^2$, $|B_T|  \approx |x|\vee R$, and $|B_T - S_{\tau}|  \approx |x|\vee R$. We now give a rough heuristic argument for why this should be the case.

By the Beurling estimate (see Lemma~\ref{beurling2}, as well as Lemma~\ref{beurling3} for the discrete version), 
the most likely place for $B$ and $S$ to leave $D$ is at some point $z$ with $|z| = O(|x|\vee R)$. Assume that $T<\tau$. Typically, under the KMT coupling, $d(S_T,\partial D) = O(\log R)$. By the Beurling estimate, the distance $S$ travels after $T$ before leaving $D$ will be of order $R^a$ with probability of order $(\log R)^{1/2}R^{-a/2}$. In that case, if $R^a=O(|x|\vee R)$, then 
$$\big|\log |B_T| - \log |S_{\tau}|\,\big| = O\left(\log \left(1+\frac{|B_T-S_{\tau}|}{|B_T|}\right)\right) = O\left(\frac{R^{a}}{|x|\vee R}\right).$$
 On the other hand, if $R^a\geq |x|\vee R$, then $\big|\log |B_T| - \log |S_{\tau}|\,\big| = O(\log R)$ (of course the constant increases with $a$ but this is taken care of easily). The value of $a$ which maximizes the product of  $\big|\log |B_T| - \log |S_{\tau}|\,\big|$ and the corresponding probability is that for which $R^a \approx |x|\vee R$, which corresponds to $|B_T-S_{\tau}|\approx |x|\vee R$. Consequently we expect that
 $$\big|\,\E^x[\log|B_T|]  - \E^x[\log|S_{\tau}|]\,\big|\approx (\log R)^{1/2} (|x|\vee R)^{-1/2}.$$

We now turn to the formal proof of Lemma~\ref{exitestimate.nowlemma}. We begin by estimating $\big|\log|B_T| - \log|S_{\tau}|\big|$, noting that this quantity can be bounded using just the information contained in $\mathcal{A}_k$ and $\mathcal{C}_m$.

If $|B_T - S_{\tau}| < |B_T|$, we can use a Taylor expansion to note that 
\begin{equation*}
\big|\log|B_T| - \log|S_{\tau}|\,\big|  
= \left|\log\left|1+ \frac{S_{\tau}-B_T}{B_T}\right|\,\right| \lesssim \frac{|B_T-S_{\tau}|}{|B_T|},
\end{equation*}
so that if $m\leq \lfloor k-1+1/\eps\rfloor$, then on $\mathcal{A}_k \cap \mathcal{C}_m$,
\begin{equation}\label{log1bound}
\big|\log|B_T| - \log|S_{\tau}|\, \big|   \lesssim R^{(m+1)\eps-1-k\eps}.
\end{equation}
On the other hand, $|B_T|\leq R^{2\eps}|B_T - S_{\tau}|$  when $m\geq \lfloor k+1/\eps\rfloor$, so by the triangle inequality, $|S_{\tau}|\lesssim R^{2\eps} |B_T - S_{\tau}|$ and another application of the inequality gives 
\begin{equation*}
\big|\log|B_T| - \log|S_{\tau}|\,\big| \leq \big|\log|B_T|\,\big|+\big|\log|S_{\tau}|\,\big|\lesssim \eps\log R+ \big|\log|B_T -S_{\tau}|\,\big|,
\end{equation*}
so that on $\mathcal{C}_m$, 
\begin{equation}\label{log2bound}
\big|\log|B_T| - \log|S_{\tau}|\,\big| \lesssim m\eps \log R.
\end{equation}

We now focus on the expected value of the indicator function in~\eqref{maineq}. The difficulty lies in finding a way to use the strong Markov property for the two processes which, under the KMT coupling, are not jointly Markov.
(See also the proof of Proposition~\ref{HM_ROC}, where we were faced with the same problem.) We first note that on $\mathcal{B}_{\ell} \cap \mathcal{H}_{r,\ell}$,
$$\sup_{0\leq t\leq R^{\ell\eps}-1}|B_t-S_t|\leq c(r+1)(\ell+1)\eps\log R.$$
We let $c$ be the constant of Theorem~\ref{kmt} and define
\begin{align*}
\xi_{r,\ell} & = \inf\big\{t\geq 0:\min\{d(B_t,\partial D),d(S_t,\partial D)\}\leq c(r+1)(\ell+1)\eps\log R\big\},\\
\mathcal{B}_{i,r,\ell}' & =  \big\{\xi_{r,\ell}\in [R^{i\eps}-1, R^{(i+1)\eps}-1)\big\}, \;\;\;0\leq i\leq \ell,\\
\mathcal{H}_{j,r,\ell}' & = \! \bigg\{\sup_{0\leq t \leq \xi_{r,\ell}}\!\!|B_t-S_t|\in [c j\log R^{(\ell+1)\eps},c(j+1)\log R^{(\ell+1)\eps})\bigg\}, \;0\leq j\leq  r,
\end{align*}
which implies
\begin{equation}\label{rareeventprob}
\PP^x(\mathcal{A}_k\cap \mathcal{B}_{\ell} \cap \mathcal{C}_m \cap \mathcal{H}_{r,\ell}) = \sum_{i=0}^{\ell}\sum_{j=0}^r \PP^x(\mathcal{A}_k\cap \mathcal{B}_{\ell} \cap \mathcal{C}_m \cap \mathcal{H}_{r,\ell}\cap \mathcal{B}_{i,r,\ell}' \cap \mathcal{H}_{j,r,\ell}'). 
\end{equation}
We note that
\begin{align}\label{maxdecomp}
\nonumber
\PP^x&(\mathcal{A}_k \cap \mathcal{B}_{\ell} \cap \mathcal{C}_m \cap \mathcal{H}_{r,\ell}\cap \mathcal{B}_{i,r,\ell}' \cap \mathcal{H}_{j,r,\ell}')\\ \nonumber
& \leq \min\bigg\{\PP^x(\mathcal{B}_{\ell} \cap \mathcal{H}_{r,\ell}), \;\PP^x(\mathcal{B}_{\ell}), \;\PP^x(\mathcal{C}_m | \mathcal{B}_{i,r,\ell}' \cap \mathcal{H}_{j,r,\ell}')\bigg\}\\ \nonumber
& \leq \min\bigg\{\PP^x\bigg(\sup_{0\leq t \leq R^{(\ell+1)\eps}}|B_t-S_t|\geq cr\log R^{(\ell+1)\eps}\bigg), \;\PP^x(\mathcal{B}_{\ell}),\\
&\qquad\qquad\qquad \PP^x(\mathcal{C}_m | \mathcal{B}_{i,r,\ell}' \cap \mathcal{H}_{j,r,\ell}')\bigg\}\nonumber\\ 
& \leq \min\bigg\{cR^{(1-r)(\ell+1)\eps}, \;\PP^x(\mathcal{B}_{\ell}),\; \PP^x(\mathcal{C}_m | \mathcal{B}_{i,r,\ell}' \cap \mathcal{H}_{j,r,\ell}')\bigg\},
\end{align}
where the last inequality follows from Theorem~\ref{kmt}.
Recall $d=\dist(x,\partial D)$ and $s=s(\ee, R, d)=(\log d)/(\ee \log R)$, so that $d=R^{s\ee}$.

If $\ell\geq 2s+3$,  then by Lemmas~\ref{constrained},~\ref{beurling2}, and~\ref{beurling3},
\begin{align*} 
\PP^x( \mathcal{B}_\ell)
& \leq \PP^x\bigg(\sup_{0\leq t \leq R^{\ell\eps}-1}|B_t-x|\leq R^{\ell\eps/2}/\log^2(R^{\ell\eps/2})\bigg)\\
&\hspace{10pt} + \PP^x\bigg(\sup_{0\leq t \leq R^{\ell\eps}}|B_t-x|\geq R^{\ell\eps/2}/\log^2(R^{\ell\eps/2}); \;T\geq R^{\ell\eps}-1\bigg)\\
&\hspace{18pt} + \PP^x\bigg(\sup_{0\leq t \leq R^{\ell\eps}-1}|S_t-x|\leq R^{\ell\eps/2}/\log^2(R^{\ell\eps/2})\bigg)\\
&\hspace{26pt} + \PP^x\bigg(\sup_{0\leq t \leq R^{\ell\eps}}|S_t-x|\geq R^{\ell\eps/2}/\log^2(R^{\ell\eps/2}); \;\tau\geq  R^{\ell\eps}-1\bigg)\\
& \leq 2R^{-c(\ell\eps/2)^2\log R} + \PP^x(B[0,\xi]\cap D^c = \emptyset) + \PP^x(S[0,\Xi]\cap D^c = \emptyset)\\
& \lesssim  R^{-c(\ell\eps/2)^2\log R} + R^{(\ell\eps/2 -s\eps)(-1/2)}\log R^{\ell\eps}\\
& \lesssim R^{s\eps/2-\ell\eps/4}\log R^{\ell\eps},
 \end{align*}
 where $\xi = \inf\big\{t\geq 0:|B_t| \geq R^{\ell\eps/2}/\log^2(R^{\ell\eps/2})\big\}$
and $\Xi = \inf\big\{t\geq 0:|S_t| \geq R^{\ell\eps/2}/\log^2(R^{\ell\eps/2})\big\}$. Note that the second inequality uses the fact that $R^{\ell\eps}-1\geq R^{\ell\eps/2}$, which for any given $\eps$ is true if $R$ is large enough, since this case assumes $\ell\geq 3$. For all other $R$, the inequality holds by making the multiplicative constant sufficiently large.

If $\ell \leq 2s -2$,  then by Lemma~\ref{largedev}, 
\begin{align*} 
\PP^x( \mathcal{B}_{\ell})  &\leq \PP^x(T\leq R^{(\ell+1)\eps}) + \PP^x(\tau\leq R^{(\ell+1)\eps})\\
&  \leq \PP^x\bigg(\sup_{0\leq t\leq R^{(\ell+1)\eps}}|B_t-x|\geq R^{s\eps}\bigg) + \PP^x\bigg(\sup_{0\leq t\leq R^{(\ell+1)\eps}}|S_t-x|\geq R^{s\eps}\bigg) \\
& \leq \PP\bigg(\sup_{0\leq t\leq R^{(\ell+1)\eps}}|B_t|\geq R^{s\eps-(\ell+1)\eps/2}R^{(\ell+1)\eps/2}\bigg)\\
& \hspace{18pt}+\PP\bigg(\sup_{0\leq t\leq R^{(\ell+1)\eps}}|S_t|\geq R^{s\eps-(\ell+1)\eps/2}R^{(\ell+1)\eps/2}\bigg) \\
&\lesssim \exp\{-R^{2s\eps-(\ell+1)\eps}/4 \},
 \end{align*}
Finally, if $2s - 1 \leq \ell \leq 2s+2 $, we have the trivial bound $\PP^x( \mathcal{B}_{\ell}) \leq 1$.
To summarize, we have
$$
\PP^x( \mathcal{B}_\ell) \lesssim
\begin{cases}
\exp\{-R^{2s\eps-(\ell+1)\eps}/4\}, & \ell\leq 2s -2,\\
1, & 2s-1 \leq \ell \leq 2s+2, \\
R^{s\eps/2-\ell\eps/4}\log R^{\ell\eps}, & \ell\geq  2s+3.
\end{cases}
$$
To evaluate $\PP^x(\mathcal{C}_m | \mathcal{B}_{i,r,\ell}' \cap \mathcal{H}_{j,r,\ell}')$, one has to be somewhat careful, as $S$ and $B$ are not jointly Markov under the KMT coupling. However, they are Markov when considered separately. We define 
the following stopping times for $S$ and $B$, respectively:
$$\tau_{r,\ell} = \inf\big\{t\geq 0:d(S_t,\bd D) \leq 2c(r+1)(\ell+1)\eps\log R\big\}$$
and 
$$T_{r,\ell} = \inf\big\{t\geq 0:d(B_t,\bd D) \leq 2c(r+1)(\ell+1)\eps\log R\big\}.$$
Then clearly, on $\mathcal{B}_{\ell} \cap \mathcal{H}_{r,\ell}$, 
$$\max\{\tau_{r,\ell},\; T_{r,\ell}\} \leq \xi_{r,\ell} \leq \min\{T,\tau\}.$$
By the Beurling estimates (see Lemmas~\ref{beurling2} and~\ref{beurling3}), we have
$$\PP^x\big(|S_{\tau_{r,\ell}}-S_t|\geq a \text{ for some }t\in [\tau_{r,\ell},\tau]\big)\lesssim a^{-1/2}((r+1)(\ell+1)\eps\log R)^{1/2}$$
and
$$\PP^x\big(|B_{T_{r,\ell}}-B_t|\geq a \text{ for some }t\in [T_{r,\ell},T]\big)\lesssim a^{-1/2}((r+1)(\ell+1)\eps\log R)^{1/2}.$$
In particular, applying the triangle inequality to 
$$S_{\tau}-B_T = (S_{\tau}-S_{\xi_{r,\ell}})+(S_{\xi_{r,\ell}}-B_{\xi_{r,\ell}})+(B_{\xi_{r,\ell}}-B_T)$$
and noting that given $\mathcal{B}_{i,r,\ell}' \cap \mathcal{H}_{j,r,\ell}'$,
$$|S_{\xi_{r,\ell}}-B_{\xi_{r,\ell}}|\leq c(j+1)(\ell+1)\eps\log R,$$
we see that the Beurling estimates imply that
\begin{align}\label{condprob} \nonumber
\PP^x&(\mathcal{C}_m | \mathcal{B}_{i,r,\ell}' \cap  \mathcal{H}_{j,r,\ell}')\\ \nonumber
& \leq \PP^x\left(\sup_{T_{r,\ell}\leq t\leq T}|B_t-B_{T_{r,\ell}}|\geq \frac{R^{m\eps}-c(j+1)(\ell+1)\eps \log R}{2}\right)\\ \nonumber
& \qquad\qquad+ \PP^x\left(\sup_{\tau_{r,\ell}\leq t\leq \tau}|S_t-S_{\tau_{r,\ell}}|\geq \frac{R^{m\eps}-c(j+1)(\ell+1)\eps \log R}{2}\right)\\ 
& \lesssim R^{-m\eps/2}((j+1)(\ell+1)\eps\log R)^{1/2}.
\end{align}

Putting~\eqref{rareeventprob},~\eqref{maxdecomp},  and~\eqref{condprob} together, we get
\begin{equation}\label{mklr}
\PP^x(\mathcal{A}_k\cap \mathcal{B}_{\ell} \cap \mathcal{C}_m \cap \mathcal{H}_{r,\ell}) \lesssim f(r,\ell,m),
\end{equation}
where 
\begin{align}\label{frlm}
f(r,\ell,&m)  =  (\ell+1) (r+1)  \notag\\ 
 &\cdot \min\big\{cR^{(1-r)(\ell+1)\eps}, \;\PP^x(\mathcal{B}_{\ell}), \;R^{-m\eps/2}((r+1)(\ell+1)\eps\log R)^{1/2}\big\}
\end{align}
and
\begin{equation*}
\PP^x( \mathcal{B}_\ell) \lesssim
\left\{
\begin{array}{ll}
\exp\{-R^{2s\eps-(\ell+1)\eps}/4\}, & \ell\leq 2s -2,\\
1, & 2s-1 \leq \ell \leq 2s+2, \\
R^{s\eps/2-\ell\eps/4}\log R^{\ell\eps}, & \ell\geq  2s+3.
\end{array}
\right.
\end{equation*}

We can now plug~\eqref{mklr},~\eqref{log1bound}, and~\eqref{log2bound} into~\eqref{maineq} to see that 
\begin{align}\label{maineqrevisited}
\big|&\,\E^x[\log|B_T|]  - \E^x[\log|S_{\tau}|]\,\big|\\
&\lesssim \sum_{k, \ell, r\geq 0} \left[ \sum_{m=0}^{\lfloor k-1+1/\eps\rfloor} R^{(m+1)\eps - 1-k\eps} f(r,\ell,m)
 +\!\! \sum_{m \geq \lfloor k+1/\eps\rfloor} \!\!m\eps f(r,\ell,m) \log R\right].\nonumber
\end{align}

Recall $d=\dist(x,\partial D)$ and $s=s(\ee)=(\log d)/(\ee \log R)$ so that $d=R^{s\ee}$.
Due to the complicated nature of the expression in~\eqref{frlm}, the estimation of the sum above is rather tedious. We show here how to find the bound for the dominant terms of the sum, that is, those for which $2s-1 \leq \ell \leq 2s+2$.  For those values of $\ell$ the sum is bounded above by
$$f(r,\ell,m)\lesssim (s+1)(r+1)\min\big\{R^{(2-2r)s\eps},\;1,\;R^{-m\eps/2}((r+1)(\ell+1)\eps\log R)^{1/2}\big\}.$$
Since $r>m/4s+1 \Rightarrow R^{(2-2r)s\eps}<R^{-m\eps/2}((r+1)(\ell+1)\eps\log R)^{1/2}$ (at least for $R$ large enough), the sum of the terms in~\eqref{maineqrevisited} where $2s-1 \leq \ell \leq 2s+2$ is bounded above (up to multiplicative constants) by 
\begin{align*}
&(s+1)\sum_{k\geq 0} \left[  \sum_{m=0}^{\lfloor k-1+1/\eps\rfloor}\right.\\
&\hspace{20pt}\bigg(\sum_{r=0}^{\lfloor m/4s+1\rfloor} (r+1)R^{(m+1)\eps -1-k\eps}R^{-(m/2)\eps}((r+1)(2s+3)\eps\log R)^{1/2}\\
&\hspace{120pt} + \sum_{r\geq \lfloor m/4s+2\rfloor} (r+1)R^{(m+1)\eps - 1 - k\eps}R^{(2-2r)s\eps}\;\bigg)\\
&\hspace{20pt}+ \sum_{m\geq \lfloor k+1/\eps\rfloor}\bigg(\sum_{r=0}^{\lfloor m/4s+1\rfloor} (r+1)m\eps R^{-(m/2)\eps}((r+1)(2s+3)\eps\log R)^{1/2}\log R\\
&\hspace{120pt}   + \left.\sum_{r\geq \lfloor m/4s+2\rfloor} (r+1)m\eps R^{(2-2r)s\eps}\log R\;\bigg)\,\right].
\end{align*}
Since the assumption $|x|\leq R^2$ implies that $d\leq R^3$, we see that
 $$s=\frac{\log d}{\eps\log R}\leq \frac{3}{\eps}.$$ 
Therefore, this multiple sum is bounded (up to a multiplicative constant which now depends on $\eps$) by
$$R^{-1/2+\eps/2}\log^4 R \lesssim R^{-1/2+\eps}.$$
One can estimate the sums of the terms for $\ell \leq 2s-2$ and $\ell \geq 2s+3$ in a
similar manner and find that they are both $o(R^{-(1/2-\eps)})$,
which proves the lemma.
\end{proof}

Having established Lemma~\ref{exitestimate.nowlemma}, we are now able to prove the main theorem.

\begin{proof}[Proof of Theorem~\ref{main_green_thm}]
We will begin with the case that $x=0$, $y \neq 0$.
As noted in Section~\ref{Sect-notation},
\[g_D(y)=\E^{y}[\log|B_{T}|] - \log|y| \;\;\; \text{and} \;\;\; G_D(y)=\E^{y}[a(S_{\tau})] - a(y).
\]
Moreover, as $|y| \to \infty$,
\[
a(y)=\frac{2}{\pi}\log|y| + k_0 + O(|y|^{-2}),
\]
so that
\[
\left| G_D(y) - \frac{2}{\pi} \, g_D(y) - k_y\right| = \frac{2}{\pi}\,\big|\,\E^y[\log|B_{T}|] -\E^y[\log|S_{\tau}|]+O(R^{-2})\,\big|. \]
In order to establish~\eqref{EJPmar29.eq1} in the case that $x=0$, $y \neq 0$, we consider $|y| \le R^2$ and $|y| \ge R^2$ separately.  If $|y| \le R^2$, then Lemma~\ref{exitestimate.nowlemma} applies and~\eqref{EJPmar29.eq1} holds. If $|y| \ge R^2$, then~\eqref{EJPmar29.eq1} follows by virtue of the fact that 
there is a constant $c < \infty$ such that if $D$ is simply connected, $V=V(D)=D \cap \Z^2$, $R=\inrad(D)$, and $|y| \ge R^2$ then
\begin{equation}\label{largebounds}
G_D(y) \le cR^{-1/2} \;\;\; \text{and} \;\;\; g_D(y) \le cR^{-1/2}.
\end{equation}
To establish~\eqref{largebounds}, write $R_y=\inrad_y(D)$, and note that $R_y\le 2|y|$. There is a constant $c_0< \infty$ such that if $R_y/8\le |w-y| \le R_y/4$ then 
$G_D(w,y) \le c_0$. Indeed, from the expression~\eqref{green1.1} for the Green's function we have
\begin{align*}
G_D(w,y) &= \E^w[a(S_{\tau}- y)]-a(y-w) \\
&\le \frac{2}{\pi}\sum_{k \ge 1}(k \log 2 + \log R_y + O(1))\, \PP^w(2^{k-1}R_y \le |S_{\tau}-y| < 2^kR_y)\\
&\qquad-\frac{2}{\pi}\log|y-w| + O(1)\\
&\le \frac{2}{\pi}\log R_y - \frac{2}{\pi}\log|y-w| +O(1) \\
&= O(1),
\end{align*}
where we used a Beurling estimate and that $\inrad_w(D) \le 5 R_y /4$ to bound the sum.
Set 
$$\eta=\min\{j \ge 0: \text{$|S_j-y| \le R_y/4$ or  $S_j \not \in D$}\}.$$
 As a function of $w$, $G_D(w,y)$ is discrete harmonic on $D
 \setminus \{y\}$, and so we see that $G_D(S_{j \land \eta},y)$ is
 a bounded martingale if $|S_0-y| > R_y/4$. Hence, if $c_0$ is as
 above,
\[
G_D(w, y) \le c_0 \, \PP^w(|S_{\eta}-y| \le R_y/4).
\] 
Set $w=0$ and note that $|S_{\eta}-y| \le R_y/4$ implies
that $|S_{\eta}| \ge R^2/2$. Consequently, since $R = \inrad(D)$, a Beurling estimate shows
the existence of a constant $c< \infty$ such that
\[
G_D(y) \le  c\, R^{-1/2}
\]
and by symmetry of the Green's function we are done. The analogue for the continuous Green's function can be proved in a similar fashion (or by using conformal invariance). This establishes~\eqref{largebounds} and shows that~\eqref{EJPmar29.eq1} holds when $|y| \ge R^2$.

To complete the proof of the theorem, we need to consider the case when $x\in V$ with $|\psi_D(x)| \le \rho$ and $y \neq 0$. In this case, if we translate $D$ to make $x$ the origin, then it follows directly from Koebe's estimate that there is a constant $c$ such that  inner radius of the translated set satisfies $\inrad_x(D) \geq cR$, and so the argument for the case when $x=0$, $y\neq 0$ implies that~\eqref{EJPmar29.eq1} holds for any $x\in V$ with $|\psi_D(x)| \le \rho$ as well.
\end{proof}

\section{Proof of Proposition~\ref{KLtheorem}}\label{prop42appendix}

In this appendix we prove Proposition~\ref{KLtheorem}. As noted in Section~\ref{MGthmsectEJP}, the proof requires a modification of a result from~\cite{KozL1} that was proved for simply connected domains with Jordan boundary. The change of setting to UBS domains requires one to establish certain technical estimates that are not immediately obvious.  In this appendix, we go over some of those results and proofs from that paper adapting them to our setting and generalizing them whenever possible. The first step in the proof of Proposition~\ref{KLtheorem} is to establish estimates for the Green's function for simple random walk in $D\cap \Z^2$ in terms of the (continuous) Green's function for $D$. This is where we require Theorem~\ref{main_green_thm}.

The following is an estimate for comparing discrete harmonic measure with continuous harmonic measure and basically says that for a UBS domain if planar Brownian motion has a chance of exiting a domain at a particular boundary arc, then simple random walk also has a chance of exiting the domain at that arc. 
For UBS domains we need to define the association between boundary subsets in a different way than for the Jordan domains used in~\cite{KozL1}. With this change, however, exactly the same proof carries through.

If $A \subset \C$ is any set, we define $\hat A$ to be the set of closed edges of $\Z^2$ that intersect $A$. That is, if  $\edgeset$ denotes the (closed) edge set of $\Z^2$, then 
\begin{equation}\label{edgedefn}
\hat A = \{ e \in \edgeset :  e \cap A \neq \emptyset\}.
\end{equation}
In the following proposition the exiting points should be understood in terms of prime ends and similarly for the harmonic measure.
\begin{proposition}\label{37}
Let $D$ be a UBS domain with $E \subset \bd D$, and let $z \in V(D)$. For all $\delta >0$, there exists a $\epsilon >0$ such that if 
$\omega(z,E,D) > \delta$, then $\PP^z(S_{T_{D_0}} \in \hat E) > \epsilon$ where $\hat E$ is as in~\eqref{edgedefn} and $\omega(z,E,D)$ denotes the continuous harmonic measure of $E$ in $D$ from $z$.
\end{proposition}

The next step is to establish several technical lemmas.  If $E\subset \C$ is any set, we define the \emph{UBS covering of} $E$ by
$$\ubs(E) = \bigcup_{\{x\in \Z^2:\bs{x}\cap E\neq \emptyset\}} \bs{x}$$
where
\begin{equation}\label{jan13.1}
\bs{z} = \{w\in\C:|\real(w)-\real(z)|<1,\, |\imag(w)-\imag(z)|<1\}
\end{equation}
as in Section~\ref{MGsect}. If $E \subset D$, where $D$ is a UBS domain, we define $\mathcal{U}_D(E)$, the UBS covering of $E$ with respect to $D$, by restricting the union in~\eqref{jan13.1} to those squares contained in $D$.

Furthermore, for $x\in D$, where $D$ is any simply connected domain, let $d_D(x) = 1-|\psi_D(x)| = \dist(\psi_D(x),\bd \Disk)$.

\begin{lemma}\label{easyestimates}
There exists a constant $c < \infty$ such that if $D$ is a simply connected UBS domain, $x,w \in \Z^2 \cap D$, and $z$, $z'\in\bs{x}$, then
\begin{itemize} 
\item[{\rm (i)}]{$d_D(z) \leq c d_D(x)$, and}
\item[{\rm (ii)}]{$|\psi_D(z)-\psi_D(w)|\leq c |\psi_D(z')-\psi_D(w)|$ for $w \notin  \overline{\bs{x}}$.}
\end{itemize}
\end{lemma}

\begin{proof}
We write $\psi=\psi_D$ and prove (i) first. All constants are considered universal unless otherwise specified. If
$d=\dist(z, \partial D)$, recall that Koebe's estimate implies
\begin{equation}\label{sept29.1}
d_D(z) \asymp d |\psi'(z)|
\end{equation}
where $\asymp$ means that each side is bounded by a constant times the other.
The result follows easily from the Koebe distortion theorem and~\eqref{sept29.1} if $\bs{x}$ is away from the boundary of $D$. (That is, if $\partial \bs{x} \cap \partial D = \emptyset$.) Hence, we may assume that $\bs{x}$ is adjacent to $\partial
D$. It is enough by \eqref{sept29.1} to prove the existence of a constant $c < \iy$ such
that
\[
d \frac{|\psi'(z)|}{|\psi'(x)|} \le c.
\]
Choose $1 < r_1 < r_2 < 3$. Write $S_1=\{z: |\text{Re}\, z - \text{Re}\, x| \le r_1, \,   |\text{Im}\, z - \text{Im}\, x| \le r_1\}$ and $S_2=\{z: |\text{Re}\, z - \text{Re}\, x| \le r_2, \,   |\text{Im}\, z - \text{Im}\, x| \le r_2\}$ for the two squares centered at $x$ with side length $2r_1$ and $2r_2$, respectively. For $j=1,2$ let $D_j$ be the (simply connected) component of $S_j \cap D$ containing $x$. Note that for $r_1$ and $r_2$ fixed, there are only finitely many possible configurations of $D_1$ and $D_2$ and the inner radii (from $x$) and diameters of $D_1, D_2$ are obviously bounded away from $0$ and from above, respectively. Moreover, the boundaries of $D_1,D_2$ are smooth except at (at most) finitely many points.  
Let $\vp$ map $D_2$ conformally onto $\D$ with $\vp(x)=0$, $\vp'(x) > 0$. Next, define
\[
h=\psi \circ \vp^{-1}:\D \to \D.\] 
(Roughly speaking, the idea is to split the behavior of $\psi$ into a local part $\vp$ which depends only on the (simple) local $O(1)$ structure of $D$ around $x$ and a global part $h$ which essentially just scales $\vp(D_1)$.) 

Note that there is a constant $\ee_1>0$ depending only on $r_2-r_1$ such that the probability that a Brownian motion started from $x$ exits $\partial D$ at a segment connecting $\partial D_1$ with $\partial D_2$ is at least $\ee_1$. (This uses the specific simple structure of $D_1,D_2$ as discussed above.) Hence, there is an $\ee>0$ (depending only on $r_2-r_1$) such that for any point $\alpha \in \vp(\overline{D_1}) \subset \mathbb{D}$ the set $\ball(\alpha, \ee) \cap \partial \mathbb{D}$ is either empty or is contained in $h^{-1}(\partial \mathbb{D})$. It follows that the Schwarz reflection principle can be applied to extend $h$ to a conformal map on $\ball(\alpha, \ee)$ for each such $\alpha$. Consequently, using Schwarz reflection and Koebe's distortion theorem, for every $\alpha \in \vp(D_1)$ we can compare derivatives along a chain of disks of radius $\ee$ to see that $|h'(\alpha)| \asymp |h'(0)|$ where the implicit constants depend only on $r_2-r_1$.

Set $\alpha=\vp(z) \in \vp(D_1)$. We have
\[
d \frac{|\psi'(z)|}{|\psi'(x)|} = d \frac{|h'(\alpha)| |\vp'(z)|}{|h'(0)|
|\vp'(x)|}.
\]
Note that $|\vp'(x)|$ is uniformly bounded away from $0$ since the inner radius from $x$ of $D_2$ at least $1$. Moreover, the Beurling estimate implies that $|\vp'(z)| \le c d^{-1/2}$ for some constant $c < \infty$. Since we already know that\[
\frac{|h'(\alpha)|}{|h'(0)|} \le c
\] for a constant $c < \infty$, (i) follows. 

To prove (ii), let $z,z' \in \mathcal{S}(x)$ and let $w \notin \overline{\bs{x}}$ be the center of a big square. We will keep the notation from the proof of (i); recall in particular the definitions of $D_1, D_2, \vp$, and $h$. By the triangle inequality,
\[
|\psi(z)-\psi(w)| \le |\psi(z)-\psi(z')| + |\psi(z')-\psi(w)|.  
\]
Let $\ell$ be the line segment connecting $\psi(z')$ with $\psi(w)$ in $\mathbb{D}$. Since $\partial D_1$ separates $w$ from $\bs{x}$ in $D$, there is a point $\xi \in \ell \cap \psi(\partial D_1)$ closest to $\psi(z')$. Let $u = \psi^{-1}(\xi)$. Since by construction $|\psi(z')-\psi(u)| \le |\psi(z')-\psi(w)|$ it is enough to show that there is a constant $c< \infty$ such that
\begin{equation}\label{nov24.1}
|\psi(z)-\psi(z')| \le c |\psi(z')-\psi(u)|.
\end{equation} 
We proceed to prove \eqref{nov24.1}. Write
\[
\alpha = \vp(z), \quad \beta=\vp(z'), \quad \gamma = \vp(u).
\]
Using again the specific simple structure of $D_1, D_2$ as discussed above, there is a $c_0 > 0$ depending only on $r_2-r_1$ such that $\dist(\vp(\mathcal{S}(x)), \gamma) > c_0$. Consequently,
\begin{equation}\label{nov24.2}
|\alpha -\beta| \le c_0^{-1} |\beta-\gamma|.
\end{equation}
Using Schwarz reflection to compare derivatives, as in the proof of (i), combined with \eqref{nov24.2}, we have that
\[
|h(\alpha) - h(\beta)| \le c |h'(\beta)||\alpha-\beta| \le c_1 |h'(\beta)||\beta - \gamma|. 
\] 
On the other hand, let $R= |h(\gamma) - h(\beta)|$ and consider the line segment $\nu(t) = h(\beta) + t(h(\gamma) - h(\beta))/R, t \in [0, R]$. Let $\Gamma(t) = h^{-1}(\nu(t))$ so that $h(\Gamma(t)) = \nu(t)$ and so $h'(\Gamma(t))\Gamma'(t) = \nu'(t)=e^{i \arg (h(\gamma)-h(\beta))}, \, t \in [0,R]$. Consequently, there is a constant $c >0$ such that
\begin{align*}
|h(\gamma)-h(\beta)|  = \left|\int_0^R h'(\Gamma(t)) \Gamma'(t) \, dt \right| 
& =  \int_0^R |h'(\Gamma(t))| |\Gamma'(t)| \, dt  \\
& \ge  c |h'(\beta)| |\Gamma| \\
& \ge c |h'(\beta)||\gamma - \beta|.
\end{align*}
Here we again compared derivatives using Schwarz reflection and Koebe's distortion theorem as before. Since $h(\alpha)=\psi(z), h(\beta)=\psi(z')$, and $h(\gamma)=\psi(u)$, the proof is complete.
\end{proof}

The following lemma shows that there is a uniform lower bound on the probability that random walk leaves the pre-image of a family of polar rectangles at each of its four sides. The centers of these polar rectangles can vary but the ratio between the angular and radial lengths is constant. It is the analogue of Lemma~3.12 in~\cite{KozL1}.

\begin{lemma}\label{app.lemma1}
Suppose $D$ is a simply connected UBS domain and that $x\in D$. There is a constant $a_0$ such that if $d_D(x) \leq a_0$, then there exist an $\eps >0$ and constants $a > 1$, $b_1< \infty$, and $b_2<\iy$ such that if
\[
\sigma \!=\! \min\{j\geq 0:S_j\not\in D \,\text{or}\; d_D(S_j)\geq ad_D(x) \text{ or } |\theta_D(S_j)-\theta_D(x)|\geq b_1d_D(x)\},
\]
then 
\begin{itemize}
\item[{\rm (i)}] $\Prob^x(S_{\sigma}\not\in D)\geq \eps$,
\item[{\rm (ii)}] $\Prob^x(d_D(S_{\sigma})\geq a d_D(x))\geq \eps$, and
\item[{\rm (iii)}] $\Prob^x(|\theta_D(S_{\sigma})-\theta_D(x)|\leq b_2 d_D(x)\big|S_{\sigma} \in D) = 1$.
\end{itemize}
\end{lemma}

\begin{proof}
Let $x \in D \cap \Z^2$ satisfy $d_D(x) \leq 1/(140c_0^3)$, where $c_0$ is the maximum of the constant from Lemma~\ref{easyestimates} and $1$. In particular, if $z$, $w$ satisfy 
$$\max\{d_D(z), d_D(w)\} \le 7c_0^3 d_D(x),$$
then we have $|\theta_D(z)-\theta_D(w)| \le 2 |\psi(z)-\psi(w)|$, where $\psi=\psi_D$. 

Note first that there is a positive probability $\delta_1>0$ (independent of $d_D(x))$ that a planar Brownian motion started from $\psi(x)$ leaves $\D$ before exiting the ball $\mathcal{B}_1:=\mathcal{B}(\psi(x), 2d_D(x))$. Consequently, by conformal invariance, a planar Brownian motion from $x$ exits  $E_1:=\psi^{-1}(\mathcal{B}_1 \cap \D)$ at $\partial D$ with probability at least $\delta_1$. Let $U_1=\mathcal{U}_D(E_1)$ be the UBS covering with respect to $D$ of $E_1$. Notice that Lemma~\ref{easyestimates}~(ii) implies that any point $z \in U_1$ satisfies
\[
|\psi(z) - \psi(x)| \le 2c_0d_D(x).
\]

By Proposition~\ref{37} there is an $\ee_1>0$ such that simple random walk from $x$ exits $U_1$ at $\partial D$ at a vertex $v$ at distance at most $1$ from a point $w$ contained in $E_1$ with probability at least $\ee_1$. By Lemma~\ref{easyestimates}~(ii), for such $v$, we have
\begin{align*}
|\theta_D(v)-\theta_D(x)| &\le 2|\psi(v)-\psi(x)| \\
& \le 2c_0|\psi(w)-\psi(x)| \\
& \le  4c_0 d_D(x). 
\end{align*}

Consider now the pre-image in $D$ of the ball around $\psi(x)$ of radius $6c_0^2d_D(x)$, namely
\[
E_2:=\psi^{-1}(\mathcal{B}(\psi(x), 6c_0^2d_D(x)) \cap \D),
\]
and let $U_2=\mathcal{U}_D(E_2)$ be the UBS covering with respect to $D$ of $E_2$. We see that $U_1 \subset U_2$.

Let $z \in U_2$. By Lemma~\ref{easyestimates}~(ii) we have that
\begin{align*}
d_D(z) & \le |\psi(z)-\psi(x)|+d_D(x) \le (6c_0^3+1)d_D(x) \\
 & \le 7c_0^3d_D(x).
\end{align*}
Similarly any $z \in U_2$ satisfies $|\theta_D(z)-\theta_D(x)| \le 12c_0^3d_D(x)$.

There is a strictly positive probability $\delta_2$ that planar Brownian motion from $\psi(x)$ exits the polar rectangle 
\[
\mathcal{R}:=\{z \in \D : |\arg(z) - \theta_D(x)| \le d_D(x), \, 1-7c_0^3 d_D(x) \le |z| \le 1\}
\]
at a point with $|z|=1-7c_0^3 d_D(x)$. Hence with probability at least $\delta_2$ planar Brownian motion from $x$ exits $\psi^{-1}(\mathcal{R}) \cap U_2$ through the ``top''; that is, at a point $w$ contained in $\partial U_2 \setminus \partial D$. Any such point satisfies (by the assumption $|\theta_D(w)-\theta_D(x)| \le d_D(x)$), 
\[
d_D(w) \ge |\psi(w)-\psi(x)|-2d_D(x) 
 \ge (6c_0^2-2)d_D(x) 
 \ge 4c_0^2 d_D(x).
\]
By Lemma~\ref{easyestimates}~(i) this means that $w$ is on the boundary of a square $\mathcal{S}(y) \subset U_2$ such that
$d_D(y) \ge 4c_0 d_D(x)$. By Proposition~\ref{37} it follows that there is an $\ee_2 > 0$ such that with probability at least $\delta_2$ a simple random walk from $x$ visits a point $y$ with $d_D(y) \ge 4 c_0 d_D(x)$ before exiting $U_2$ (and thus before exiting $\{z: |\theta_D(z)-\theta_D(x)| \le 12c_0^3d_D(x)\}$).
Finally, using part~(ii) of Lemma~\ref{easyestimates} for the estimate on $b_2$, the desired conclusion follows by taking $a_0=1/(140c_0^3)$, $\ee=\min\{\ee_1, \ee_2\}$, $a=4c_0, b_1=12c_0^3$, and $b_2=24c_0^4$.
\end{proof}

The final preliminary result that is needed in order to prove Proposition~\ref{KLtheorem} is the analogue of Corollary~3.15 of~\cite{KozL1}. Assuming Lemma~\ref{app.lemma1}, the proof in the UBS setting is identical to the proof in the original Jordan setting. We refer to Proposition~3.14 and Corollary~3.15 of~\cite{KozL1} for more details. 

For any $a \in (0,1/2)$ and for any $\theta_1 < \theta_2$, 
let $\xi_D(a,\theta_1,\theta_2)$ be the first
 time $t \geq 0$ that a random walk leaves the pre-image of the polar rectangle
\[\{y \in V(D): d_D(y) \leq a, \; \theta_1 \leq \theta_D(y) \leq \theta_2\}.\]
Consider the probability that the random walk conditioned not to exit the polar rectangle at $\partial D$ exits at the ``top'': 
\[q_D(x,a,\theta_1,\theta_2) = \Prob^x \left(d_D(S_{\xi_D(a,\theta_1,\theta_2)}) > 
a \mid S_{\xi_D(a,\theta_1,\theta_2)} \in V(D)\right),\]
and note that if $\theta_1 \leq \theta_1' \leq \theta_2' \leq 
\theta_2$, then $q_D(x,a,\theta_1',\theta_2')\leq  q_D(x,a,\theta_1,\theta_2)$.
\begin{corollary}\label{apr8.cor1}
There exist $c$ and $\beta$ such that if $a \in (0,1/2)$, $r > 0$, $D$ is a UBS domain, and $x \in V(D)$, then
$$q_D(x,a,\theta_D(x) - ra, \theta_D(x) + ra) \geq 1 - c e^{-\beta r}.$$
\end{corollary}

We can now complete the proof of Proposition~\ref{KLtheorem}. Given the technical results that we have just discussed, the proof essentially follows as the proof of Proposition~3.10 of~\cite{KozL1}. (Unlike in that paper, however, we are not considering any two arbitrary points in the domains but rather one point near the boundary and one point near the origin.) Consequently, we will not give all the details in the proof below, but rather highlight the steps that affect the rate and show how the exponent of $1/4$ occurs. We will begin by assuming that $p \in (0,1/2)$ is arbitrary, and then we will derive a number of estimates in terms of $p$. As various steps in the proof additional restrictions on $p$ will be added, and at the end we will optimize to find $p=1/4$. 

Before we give the details, let us give a heuristic argument ignoring the particular error terms. If $D^*=\{x \in V :  g_D(x) \ge c_0R^{-p}\}$ then  $g_D(z) \approx R^{-p}$ for $z \in \partial D^*$. By Theorem~\ref{main_green_thm} the same is true for a constant times $G_D(z)$. Let $y \in V \setminus D^*$. If $\eta$ is the first time $S$ exits $V\setminus D^*$, we can write $G_D(y)  \approx  \text{const.}\,R^{-p}\Prob^y(S_{\eta} \in D^*)$. Also, by Corollary~\ref{apr8.cor1}, a simple random walk from $y$ that hits $D^*$ before exiting $V$ is likely to do so without the argument of its conformal image in $\D$ changing too much: $\Prob^y(S_{\eta} \in K(y)) \approx \Prob^y(S_{\eta} \in D^*)$ where $K(y)$ is a suitable subset of $D^*$ where $e^{i\theta(z)} \approx e^{i\theta(y)}$. Note that 
\[
g_{\D}(u,v) \approx g_{\D}(v) \frac{1-|u|^2}{|u-e^{i\theta(v)}|}
\]
for $v$ close to the boundary. Hence, by using Theorem~\ref{main_green_thm} and conformal invariance of the Green's function, we expect that 
\begin{align*}
G_D(x,y) &\approx \sum_{z \in K(y)} G_D(x,z) \Prob^y(S_{\eta} = z) \\
 &\approx \sum_{z \in K(y)} \frac{2}{\pi}\, g_D(x,z) \Prob^y(S_{\eta} = z) \\
& \approx \sum_{z \in K(y)}  \frac{2}{\pi}\, g_D(z) \frac{1-|\psi_D(x)|^2}{|\psi_D(x) - e^{i\theta(z)}|^2}\Prob^y(S_{\eta} = z) \\
& \approx \frac{2}{\pi}\, R^{-p}\Prob^y(S_{\eta} \in D^*)  \frac{1-|\psi_D(x)|^2}{|\psi_D(x) - e^{i\theta(y)}|^2} \\ 
& \approx G_D(y) \frac{1-|\psi_D(x)|^2}{|\psi_D(x) - e^{i\theta(y)}|},
\end{align*}  
for suitable error terms depending on $p$.

\begin{proof}[Proof of Proposition~\ref{KLtheorem}]
Let $D$ be a simply connected UBS domain, write $R=\inrad(D)$, and let $V=V(D) = D \cap \Z^2$. Suppose further that $0<p<1/2$ and note that if $y \in V$ with $|\psi_D(y)| \ge 1 - R^{-p}$, then there exists some universal constant $c_0$ such that $g_D(y) < c_0R^{-p}$. Using this $c_0$, let
$$D^{*} = \{x \in V : g_D(x) \ge c_0R^{-p}\}.$$
Finally, let 
\[\eta = \eta(D) = \min\{j \ge 0: S_j \in D^{*} \cup V^c\}.\]

Let $z$, $w \in V$ with $z \in D^{*}$, $w \not \in D^{*}$, and $|z-w|=1$ so that 
\begin{equation}\label{p42eq1}
g_D(w) < c_0R^{-p} \leq g_D(z).
\end{equation}
However, even though $z \in D^{*}$, the Green's function $g_D(z)$ cannot be that much larger than $c_0R^{-p}$ since $|z-w|=1$ for some $w \not \in D^{*}$. We will now quantify this statement.
If we set $u=\psi_D(z)$ and let $f=\psi_D^{-1}$, then the  Koebe distortion theorem implies that 
$$|f'(u)| \ge c(1-|\psi_D(z)|) |f'(0)| \ge cR^{-p}|f'(0)|.$$
From the Koebe one-quarter theorem we know $|f'(0)| \ge \inrad(D)/4 = R/4$ and so
\begin{equation}\label{p42eqdist}
|f'(u)| \ge cR^{1-p}.
\end{equation}
The Koebe one-quarter theorem thus implies that
$\ball(z, R^{-p}|f'(u)|/4) \subset D$, and so we conclude from~\eqref{p42eqdist} that
$\ball(z, cR^{1-2p}) \subset D$.  Since $g_D$ is a positive bounded harmonic function in $\ball(z, cR^{1-2p})$ we can use Exercise~2.17 of~\cite{SLEbook} to conclude
$$g_D(z)=g_D(w)+O(R^{-(1-2p)})$$
assuming that the error term is not larger than the leading term which is true as long as $p<1-2p$. Thus, we have introduced a restriction on $p$, namely that $0<p<1/3$. Combined with~\eqref{p42eq1} we conclude that 
\[g_D(z)=c_0R^{-p}+O(R^{-(1-2p)}).\]

Using Theorem~\ref{main_green_thm}, it now follows that for any $0<\eps<1/2$,
$$G_D(z) = (2/\pi)\, c_0R^{-p} +O(R^{-(1-2p)}) +O(|z|^{-2}) + O(R^{-(1/2-\eps)}).$$
Thus, since $\dist(z, V \setminus D^{*}) =1$, it follows that $|z|>R^{1/4}$ which implies
$$G_D(z) = (2/\pi)\, c_0R^{-p} + O(R^{-(1/2-\eps)}) = (2/\pi)\, c_0R^{-p}[1 + O(R^{p-(1/2-\eps)})]$$
and similarly for $G_D(w)$. We also find another restriction on $p$, namely that $p-(1/2-\eps)<0$. Therefore, using the strong Markov property, for any $y \in V \setminus D^{*}$, 
\begin{equation} \label{sept13.3}
G_D(y) = (2/\pi)\, c_0R^{-p}\, \Prob^y(S_\eta \in D^{*}) \, [1 + O(R^{p-(1/2-\eps)})]. 
\end{equation}

In a similar fashion, note that if $x\in V$ with $|\psi_D(x)| \le \rho$ and 
 $z\in D^{*}$, then $g_D(x,z) \ge cR^{-p}$ for some $c$,
  and hence by Theorem~\ref{main_green_thm},  if $|z-x|\geq R^{1/4}$, then
 \begin{equation*}
 G_D(x,z) = (2/\pi)\, g_D(x,z) \, [1 + O(R^{p-(1/2-\eps)})].
 \end{equation*}
If $x \in V$ with $|\psi_D(x)| \le \rho$ and $y \in V\setminus D^{*}$, then there exists 
some $0<\delta<1$ such that $|\psi_D(x) - \psi_D(y)| \geq \delta$. 

Using~\eqref{defngreen} and the fact that the Green's function is conformally invariant, one can check that if $\zeta'=\psi_D(y)=(1-r)
 e^{i\theta}$ and $\zeta=\psi_D(x) \in \Disk$ with $|\zeta-\zeta'|\geq r$, then
\begin{equation}\label{oct30.1}
g_D(x,y) = g_{\D}(\zeta,\zeta') = \frac{g_{\D}(\zeta')\,(1-|\zeta|^2)}{|\zeta-e^{i \theta}|^2}\, 
\left[1+O\left(\frac{r}{|\zeta-\zeta'|}\right)\right].
\end{equation}

Let $y \in V \setminus D^*$ and let $z \in D^*$ be as above with one neighbor in $V \setminus D^*$. Assume also that $z \in K=K(y):=\{z \in D^{*}, \, |\psi_D(z) - \psi_D (y)| \leq c_1 R^{-p} \log R\}$. Then~\eqref{oct30.1} and the estimate on $g_D(z)$ imply that
\begin{equation*}
g_D(x,z) 
= \frac{c_0R^{-p} \,(1-|\psi_D(x)|^2)}{|\psi_D(x)-e^{i\theta_D(y)}|^2}\,
\left[\, 1 + O(R^{-p}  \log R ) \,\right].
\end{equation*}

It now follows from Corollary~\ref{apr8.cor1} that there exists a $c_1$ such that if 
$$\xi = \xi(D, c_1) = \min \{ j \ge 0: S_j \not\in V \,\, \text{or} \,\, |\psi_D(S_j) - \psi_D(y)| \ge c_1 R^{-p}\log R\},$$
then
$$\Prob^y(\xi < \eta) \le c_1 R^{-5} \Prob^{y}(S_{\eta} \in D^{*} ).$$
In particular
\[
\Prob^{y}(S_{\eta} \in K(y))=(1-O(R^{-5}))\Prob^{y}(S_{\eta} \in D^{*}).
\]

Hence, by stopping the random walk from $y$ when it hits $D^{*}$  
it can be shown that
\begin{align*}
G_D(x,y)&=\Prob^y(S_\eta \in D^{*})\, \frac{(2/\pi)\, c_0R^{-p}\, (1-|\psi_D(x)|^2)}
{|\psi_D(x)-e^{i\theta_D(y)}|^2}\\
&\qquad\cdot [\,1 + O(R^{p-(1/2-\eps)}) \,]\cdot [\, 1 + O(R^{-p} \, \log R )\, ].
\end{align*}
Combining this with~\eqref{sept13.3} gives
\begin{align*}
\frac{G_D(x,y)}{G_D(y)}&= \frac{1-|\psi_D(x)|^2}
{|\psi_D(x)-e^{i\theta_D(y)}|^2} \cdot
\frac{ [\, 1 + O(R^{p-(1/2-\eps)}) \,] \!\cdot\! [\, 1 + O(R^{-p} \, \log R)\, ]} 
{[\,1+O(R^{p-(1/2-\eps)})\,]}\\
&= \frac{1-|\psi_D(x)|^2}
{|\psi_D(x)-e^{i\theta_D(y)}|^2} \cdot
[\, 1 + O(R^{p-(1/2-\eps)}) \,] \!\cdot\! [\, 1 + O(R^{-p} \, \log R)\, ].
\end{align*}
Thus, solving $p-(1/2-\eps)=-p$ for any $0<\eps<1/2$ gives $p<1/4$ and so choosing $p=1/4-\epsilon'$ for 
any $0<\epsilon'<1/4$ yields~\eqref{KLestimate} completing the proof.
\end{proof}

\subsection*{Acknowledgements}
The authors would like to express their gratitude to the Banff
International Research Station for Mathematical Innovation and
Discovery (BIRS) and the Mathematisches Forschungsinstitut Oberwolfach
(MFO) where much of this work was carried out. In particular, the authors benefited from participating in a \emph{Research in Teams} programme at BIRS from August~24 to August~31, 2008, and at a \emph{Research in Pairs} programme at MFO from August~2 to August~15, 2009. Thanks are also owed to Robert Masson for useful discussions. Finally, gratitude is also due to an anonymous referee for excellent comments and suggestions based on a thorough reading of our paper.

Bene\v{s} is supported in part by PSC-CUNY Awards \#60130-3839 and \#62408-0040. This work was carried out while Johansson Viklund was supported by the Simons Foundation, the Knut and Alice Wallenberg 
Foundation, and in part by the Swedish Royal Academy of Sciences. Kozdron is supported in part by the Natural Sciences and Engineering Research Council (NSERC) of Canada.


\end{document}